\documentclass[aap,preprint]{imsart}

\RequirePackage{amsthm,amsmath,amsfonts,amssymb}
\RequirePackage[numbers]{natbib}
\RequirePackage[colorlinks,citecolor=blue,urlcolor=blue]{hyperref}
\RequirePackage{graphicx}

\startlocaldefs
\theoremstyle{plain}

\newtheorem{theorem}{Theorem}[section]

\newtheorem{corollary}[theorem]{Corollary}
\newtheorem{proposition}[theorem]{Proposition}
\newtheorem{assumption}[theorem]{Assumption}
\newtheorem{optimisation}[theorem]{Optimisation}
\theoremstyle{remark}
\newtheorem{definition}[theorem]{Definition}
\newtheorem*{example}{Example}

\usepackage{dsfont}
\usepackage[colorinlistoftodos]{todonotes}
\usepackage{enumitem}
\usepackage{booktabs}
\usepackage{multirow}
\usepackage{lscape}
\usepackage{wrapfig}
\usepackage{bm}
\usepackage{graphicx}

\usepackage[linesnumbered,boxruled,lined,noend]{algorithm2e}
\SetAlFnt{\footnotesize}
\SetKwComment{tcp}{$\triangleright$\ }{}
 \allowdisplaybreaks

\newcommand{\C}{{\mathbb{C}}}
\newcommand{\E}{{\mathbb{E}}}
\newcommand{\Q}{{\mathbb{Q}}}
\newcommand{\V}{{\mathbb{V}}}
\renewcommand{\P}{{\mathbb{P}}}
\newcommand{\Z}{{\mathbb{Z}}}

\newcommand{\ba}{\bm{a}}
\newcommand{\bX}{\boldsymbol{X}}
\newcommand{\bY}{\bm{Y}}
\newcommand{\bx}{\bm{x}}
\newcommand{\bW}{\bm{W}}
\newcommand{\bC}{\bm{C}}
\newcommand{\bB}{\bm{B}}
\newcommand{\bmu}{\bm{\mu}}
\newcommand{\btmu}{\bm{\tilde\mu}}
\newcommand{\bnu}{\bm{\nu}}
\newcommand{\blambda}{{\bm{\lambda}}}
\newcommand{\bWlambda}{\bm{W^\lambda}}
\newcommand{\bdelta}{\bm{\delta}}
\newcommand{\bc}{{\bm{c}}}
\newcommand{\bd}{\bm{d}}
\newcommand{\bz}{{\bm{z}}}

\newcommand{\repone}{{r_\ep}}
\newcommand{\rep}{{\bm{r}_\ep}}

\newcommand{\Rep}{{R_\ep}}

\newcommand{\ostar}{\mathbin{\mathpalette\make@circled\star}}
\newcommand{\make@circled}[2]{%
  \ooalign{$\m@th#1\smallbigcirc{#1}$\cr\hidewidth$\m@th#1#2$\hidewidth\cr}%
}
\newcommand{\smallbigcirc}[1]{%
  \vcenter{\hbox{\scalebox{0.77778}{$\m@th#1\bigcirc$}}}%
}

\newcommand{\blog}{{\text{\textbf{log}}}}
\newcommand{\Qlh}{{\mathbb{Q}_{\blambda,\bm h}}}

\newcommand{\mfX}{\mathfrak{X}}

\newcommand{\Id}{{\mathds{1}}}
\newcommand{\R}{{\mathds{R}}}

\newcommand{\tT}{{t\in[0,T]}}

\newcommand{\D}{{\mathcal{D}}}
\newcommand{\F}{{\mathcal{F}}}

\newcommand{\mcP}{{\mathcal{P}}}
\newcommand{\Lg}{{\mathcal{L}}}
\newcommand{\mcQ}{{\boldsymbol{\mathcal{Q}}}}

\newcommand{\eone}{{\bm {\eta}_1}}
\newcommand{\etwo}{{\bm {\eta}_2}}
\newcommand{\estarone}{{\bm{\eta}^*_1}}
\newcommand{\estartwo}{{\bm{\eta}^*_2}}
\newcommand{\etaep}{{\boldsymbol{\eta}^*_\ep}}
\newcommand{\etaepone}{{\eta^*_\ep}}

\newcommand{\bg}{{\bm g}}

\newcommand{\bh}{{\bm h}}

\newcommand{\bmff}{{\bm{\mathfrak{f}}}}
\newcommand{\mff}{{\mathfrak{f}}}
\newcommand{\mfD}{{\mathfrak{D}}}

\newcommand{\mR}{{\mathcal{R}}}

\newcommand{\ep}{\varepsilon}

\newcommand{\VaR}{{\text{VaR}}}
\newcommand{\TVaR}{{\text{TVaR}}}

\newcommand{\dQP}{\frac{\diff\Q}{\diff\P}}

\newcommand{\norm}[1]{\left\lVert#1\right\rVert}
\newcommand\abs[1]{\left|#1\right|}

\newcommand{\diff}{\mathrm{d}}

\newcommand{\EDt}{{\mfD^{\bmff, \bdelta}_{t}}}

\newcommand{\EDo}{{\mfD^{\bmff, \bdelta}_{0}}}

\DeclareMathOperator{\Tr}{Tr}

\endlocaldefs

\begin{document}

\begin{frontmatter}
\title{Minimal Kullback-Leibler Divergence for Constrained L\'evy-It\^o Processes}
\runtitle{Minimal KL for Constrained L\'evy-It\^o Processes}

\begin{aug}
\author[A,C]{\fnms{Sebastian}~\snm{Jaimungal}\ead[label=e1]{sebastian.jaimungal@utoronto.ca}\orcid{0000-0002-0193-0993}},
\author[A]{\fnms{Silvana M.}~\snm{Pesenti}\ead[label=e2]{silvana.pesenti@utoronto.ca}\orcid{0000-0002-6661-6970}}
\and
\author[B]{\fnms{Leandro}~\snm{S\'anchez-Betancourt}\ead[label=e3]{leandro.sanchez-betancourt@kcl.ac.uk}\orcid{0000-0001-6447-7105}}
\address[A]{Department of Statistical Sciences, University of Toronto\printead[presep={,\ }]{e1,e2}}

\address[B]{Department of Mathematics, King's College London\printead[presep={,\ }]{e3}}
\end{aug}
\address[C]{Oxford-Man Institute of Quantitative Finance, University of Oxford\vspace{2em}}
\today

\begin{abstract}
Given an $n$-dimensional stochastic process $\bX$ driven by $\mathbb{P}$-Brownian motions and Poisson random measures, we seek the probability measure $\mathbb{Q}$, with minimal relative entropy to $\mathbb{P}$, such that the $\mathbb{Q}$-expectations of some terminal and running costs are constrained. We prove existence and uniqueness of the optimal probability measure, derive the explicit form of the measure change, and characterise the optimal drift and compensator adjustments under the optimal measure. 
We provide an analytical solution for Value-at-Risk (quantile) constraints, discuss how to perturb a Brownian motion to have arbitrary variance, and show that pinned measures arise as a limiting case of optimal measures. The results are illustrated in a risk management setting -- including an algorithm to simulate under the optimal measure -- and explore an example where an agent seeks to answer the question: what dynamics are induced by a perturbation of the Value-at-Risk and the average time spent below a barrier on the reference process?
\end{abstract}


\begin{keyword}
 \kwd{Relative entropy}
 \kwd{Kullback-Leibler}
 \kwd{L\'evy-It\^o processes}
 \kwd{Reverse sensitivity}
 \kwd{Risk Management}
 \kwd{Model Uncertainty}
 \kwd{Cryptocurrency}
\end{keyword}

\end{frontmatter}

\section{Introduction}
We consider stochastic processes that follow L\'evy-It\^o dynamics under a reference probability measure $\P$ over a finite time horizon. The reference measure may arise in a data driven way and / or from modelling assumptions, however, it does not precisely capture all probabilistic beliefs of a modeller. In this work, misspecification under $\P$ are characterised via expected values of functions of the stochastic process at terminal time and expected running costs of the processes over the entire time horizon. To mitigate model error, we seek over all absolutely continuous probability measures, under which the process satisfies these constraints, the one which is closest to the reference measure $\P$ in relative entropy, also called Kullback-Leibler (KL) divergence. Thus, the key contribution of this work is solving the following constrained optimisation problem: Find the probability measure(s) that has minimal KL-divergence subject to constraints that can be written as (i) expected values of functions applied to the stochastic process at terminal time, and (ii) expected running costs of the processes over the entire time horizon. 

We proceed to solve the optimisation problem by first considering a related optimisation problem where we seek over a subset of probability measures. Specifically, the subset consists of equivalent probability measures that arise from Dol\'eans-Dade exponentials and we study this related problem using stochastic control techniques for L\'evy-It\^o processes. That is, we solve the dynamic programming equations and characterise a candidate solution (Proposition \ref{prop:candidate}), prove that the candidate solution is indeed the value function associated with the optimisation problem and that the resulting controls (which induce the optimal measure change) are admissible (Theorems \ref{thm: verification} and \ref{thm: E! of lagrange}).
Furthermore, we show that the optimal measure change can be written as the exponential of a collection of random variables corresponding to the constraints (Corollary \ref{cor:dQ-dP} and Theorem \ref{thm: E! of lagrange}). Finally, we prove that if a solution to the sub-problem (seeking over the subset of equivalent measures) exits, then  it is unique and, moreover, it is the unique solution to the original optimisation problem, where we seek over all absolutely continuous probability measures (Theorem \ref{thm:P-and-Pprime}). 

We illustrate the dynamics of the stochastic process under the optimal measure using multiple examples. For example, for the case of two Value-at-Risk (also known as quantile) constraints, we provide an analytical expression for the optimal measure change. We further show how to optimally change the dynamics of a Brownian motion to have zero mean and arbitrary variance. We also discuss the connection of the solution to our constrained optimisation problem to pinned measures; probability measures where the terminal value of the process lies almost surely within a Borel measurable set. While such measures are not equivalent with respect to $\P$ and thus do not fall into the set of admissible measures of our related problem, we derive them as a limiting case of solutions to our optimisation problem. We further consider infinitesimal perturbations; that is, we solve the problem where the constraints are equal to their $\P$-expectation plus $\ep$ multiplied with a direction $\bdelta$. In this setup, we prove  that the Lagrange multiplier is, up to order $\ep$, the inverse of the $\P$-covariance matrix of the constraint functions multiplied by $\ep$ and the direction of the perturbation. Using this result, we define a derivative -- termed \textit{entropic derivative} -- of a risk functional along constraints in direction of least relative entropy. As examples we show the connection of the entropic derivative to differential sensitivities of risk functionals such as the Tail-Value-at-Risk and distortion risk measures. Finally, we provide an algorithm for solving and simulating from the optimal probability measure and illustrate the numerical results on a running cost constraint in a financial setting on real data.

Studying minimal relative entropy subject to constraints has a long history starting with the seminal paper of \cite{Csiszar1975divergence}. Applications to model risk assessment include
\cite{Glasserman2014QF} which uses the relative entropy to quantify worst-case model errors in a static setting. Similarly and also in a static setting, \cite{Breuer2016MF} proposes to quantify distributional model risk by considering alternative models that lie within a KL-tolerance distance from a reference measure. The work in \cite{lam2016robust} investigates what happens in the limit of small KL-tolerances. Conceptually close to our work -- though in a static setting -- is \cite{Pesenti2019EJOR} which considers a reference probability measure and finds the probability measure that satisfies risk measure constraints with minimal relative entropy to the reference measure. None of these works, however, consider stochastic processes and thus do not consider running cost constraints.

The KL-divergence has many applications in financial mathematics. Starting with the influential work of \cite{schweizer1992mean}, the vast majority of the literature on minimising relative entropy focuses on its application for derivative pricing in incomplete markets. To avoid arbitrage, such questions require restricting to martingale measures. Articles \cite{avellaneda1997calibrating,avellaneda1998minimum}, for example, consider a reference model and seek over all equivalent martingale measures, in a simple diffusive setting, to ensure that a collection of prices of European contingent claims are matched correctly. Article \cite{cont2004nonparametric} extends \cite{avellaneda1998minimum} by using a  compound Poisson process (with discrete jump sizes) as a reference model. The work in 
\cite{Jeanblanc2007AAP} studies the problem of finding martingale measures for exponential L\'evy processes that minimise R\'enyi- and KL-divergences, and \cite{de2012convex} uses convex regularisation techniques, motivated by KL-divergence as a regulariser, to calibrate local volatility models. In this exposition, we consider a different problem in that we do not restrict to martingale measures but solve for the optimal dynamics of the process such that given constraints are fulfilled. In particular, in contrast to our setup, all of the above mentioned literature work with  risk-neutral measures (i.e., martingale measures). Moreover, the running cost constraint considered in this work is novel. A natural interpretation of the running costs constraint in mathematical finance is that of the average time spent below a barrier which we consider in the numerical example section.

Optimising the (relative) entropy has a long tradition and many applications in physics. Article \cite{jaynes1957information} for example investigates the problem of specifying expectations of observables (random variables) and seek over distributions (models) that match these expectations, and which maximise the Shannon entropy to obtain the model that best reflects the information contained in the expectations. This work has been extended in many directions, and for instance \cite{parrondo2009entropy} shows how relative entropy may inform about the arrow of time by looking at the relative entropy between the distribution of a process forward in time and its reversed version. As another example, \cite{shalymov2017dynamics} proposes a process for how a (physical) system may evolve to a state of minimal relative entropy, subject to an energy and mass constraint, based on the speed-gradient principle (see, e.g., \cite{fradkov2008speed}).

Calculations of the KL-divergence of processes has been studied by \cite{vajda1990distances}, which establishes that $f$-divergences, and hence the KL-divergence, between two probability measures on path space may be approximated by focusing on their finite dimensional distributions. An application to uncertainty quantification in a dynamic setting is 
\cite{dupuis2020sensitivity}, which uses a variational representation of the R\'enyi-divergence which encompasses the KL-divergence to provide uncertainty quantification bounds for rare events.

This paper is structured as follows. Section \ref{sec:model-setup} introduces the necessary notation and the two constraint optimisation problems we consider. We present a formal derivation of a candidate solution and a verification theorem in Section \ref{sec:verification}. Section \ref{sec:RN-representation} contains an alternative representation of the associated Radon-Nikodym derivatives and in Section \ref{sec:solution-opt} we state the existence and uniqueness of the solutions to both optimisation problems. Examples including analytical solution for Value-at-Risk (quantile) constraints, Brownian motion with arbitrary variance, and the connection of the solution to our optimisation problem to pinned measures are discussed in Section \ref{sec:examples}.
In Section \ref{sec:small-perturbation-lagrange} we consider infinitesimal perturbations, that is the optimisation problem where the constraints  their $\P$-expectations plus $\ep$ multiplied by a direction $\bdelta$, and derive the optimal Lagrange multiplier up to order $\ep$. 
In Section \ref{sec:entropic-derivative} we define the entropic derivative an relate it to differential sensitivities of risk functionals.
Section \ref{sec:numerical-ex} proposes an algorithm for calculating the dynamics of the process under the optimal measure, which we illustrate on a financial dataset and a running cost constraint.

\section{Optimisation Problem and its Solution}
\subsection{Model Setup and Optimisation Problems}\label{sec:model-setup}
We work on a complete filtered probability space  $(\Omega, \P, \F, \{\F_t\}_\tT)$ with time horizon $T>0$, and refer to $\P$ as the physical (or real-world) probability measure. On this space we introduce families of so-called L\'evy-It\^o processes. For aspects of the theory of such processes see \cite{applebaum2009levy,oksendal2019applied,bouzianis2021levy}. Here, we consider an $m$-dimensional $\P$-Brownian motion $\bW=(W^1,\dots,W^m)^\intercal$ and $l$ independent Poisson random measures (PRM) $\bmu(\diff t,\diff \bm z) = (\mu_1(\diff t, \diff z_1),\dots,\mu_l(\diff t, \diff z_l))^\intercal$, $t \in [0,T]$, $\bm z = (z_1,\dots,z_l)^\intercal$, associated with $l$ one-dimensional independent  L\'evy processes with finite second moments for all $\tT$. Further, we denote by $\bnu(\diff t, \diff \bm z) = (\nu_1(\diff t, \diff z_1),\dots,\nu_l(\diff t, \diff z_l))^\intercal$ the compensator of $\bmu$ and by $\btmu = \bmu-\bnu$ the compensated measure.\footnote{In the present framework $\bnu(\diff t, \diff \bm z)$ can be written as $\bnu(\diff \bm z)\diff t$ and we use them interchangeably.} That is, for any $i \in \D:= \{1, \ldots, l\}$, $\nu_i$ is the compensator associated with $\mu_i$ and $\tilde\mu_i = \mu_i -\nu_i$ the compensated random measure under $\P$. 

We consider an $n$-dimensional stochastic process $\bX:= (\bX_t)_\tT$  starting at $\bX_0 = \bx_0\in\R^n$ and which evolves according to the stochastic differential equation (SDE) under $\P$
\begin{equation}\label{eq: X dynamics ndim}
    \diff \bX_t = \bm\alpha(t,\bX_t)\,\diff t+\bm\sigma(t,\bX_t)\,\diff \bW_t + \int_{\R^l}\bm \gamma(t,\bX_{t^-},\bm z)\,\btmu(\diff t, \diff \bm z)\,,
\end{equation}
where $\bm\alpha:[0,T]\times\R^n \to \R^n$, $\bm\sigma:[0,T]\times\R^n \to \R^{n\times m}$, and $\bm\gamma:[0,T]\times\R^n\times\R^l \to \R^{n\times l}$ satisfy the standing Assumption \ref{assumption: lip and linear growth} below. 
Equation \eqref{eq: X dynamics ndim} is the matrix notation meaning that the $i$-th component of $(\bX_t)_{t\in [0,T]}$ satisfies the SDE under $\P$
\begin{align*}
    \diff X^i_t = \alpha_i(t,\bX_t)\,\diff t + \sum_{j=1}^{m} \sigma_{ij}(t,\bX_t)\,\diff W^j_t + \sum_{j\in \D}\; \int_{\R}\; \gamma_{ij}(t,\bX_{t^-}, z_j)\,\tilde\mu_j(\diff t,\diff z_j)\,
\end{align*}
with $\bm\alpha = (\alpha_1,\dots,\alpha_n)^\intercal$, $\bm\sigma = [\sigma_{ij}]_{i,j}$, and $\bm\gamma = [\gamma_{ij}]_{i,j}$, for $i \in \{1, \ldots, n\}$ and $j\in \{1, \ldots, m\}$. For $j \in \D$, we use the notation $\bm\gamma^{(j)} = (\gamma_{1,j}, \ldots \gamma_{n,j})$ to refer to the $j$-th column of $\bm\gamma$. Furthermore, we assume that each column $\bm\gamma^{(j)}$, $j\in\D$, depends on $\bm z$ only through $z_j$, i.e., $\bm\gamma^{(j)}(t,\bx,\bm z) \equiv \bm\gamma^{(j)}(t,\bx, z_j)$.

Throughout we use the following notation. For a function $\ell \in\mathcal{C}^{1,2}([0,T]\times \R^n; \R)$ we write $\nabla_{\bx} \ell$ for the vector of its partial derivatives and $\nabla^2_{\bx} \ell$ for the Hessian matrix of (mixed) second derivatives. We further define 
\begin{subequations}\label{eq:def-Delta-z}
\begin{align}
    \bm\Delta_{\bm z} \ell(t,\bx) 
    &= \left(\Delta^1_{z_1} \ell(t,\bx),\dots, \Delta^l_{z_l}\ell(t,\bx)\right)\,,
    \quad \text{where}
    \\
    \Delta^j_{z_j} \ell(t,\bx) 
    &:= \ell\left(t,\,\bx+\bm\gamma^{(j)}(t,\bx, z_j)\right)-\ell(t,\bx)\,,
    \quad  j\in\D \,.
\end{align}
\end{subequations}

The next assumption guarantees that the stochastic process $(\bX_t)_{t\in [0,T]}$ given in \eqref{eq: X dynamics ndim} is well-defined.
\begin{assumption}\label{assumption: lip and linear growth}
The functions $\bm\alpha:[0,T]\times\R^n \to \R^n$, $\bm\sigma:[0,T]\times\R^n \to \R^{n\times m}$, and $\bm\gamma:[0,T]\times\R^n\times\R^l  \to \R^{n\times l}$  satisfy the usual linear growth and Lipschitz continuity conditions. That is for all $t \in [0,T]$ and $\bx\in\R^n$ there exists $C_1<\infty$ such that 
\begin{equation*}
    \norm{\bm\sigma(t,\bx)}^2 
    +
    |\bm\alpha(t,\bx)|^2
    +
    \int_{\R}\;\sum_{j\in\D}\; |\bm\gamma^{(j)}(t,\bx, z_j)|^2\nu_j(\diff z_j) \leq C_1\,(1+|\bx|^2)\,,
\end{equation*}
where $|\cdot|$ is the Euclidean norm and $\norm{\bm \sigma}^2 = \sum_{ij}\sigma_{ij}^2$ the Frobenius norm.
Moreover, for all $t \in [0,T]$, $\bx\in\R^n$, and $\bm y \in\R^n$ there exists $C_2<\infty$ such that 
\begin{equation*}
\begin{split}
    \norm{\bm\sigma(t,\bx) - \bm\sigma(t,\bm y)}^2 &+ |\bm\alpha(t,\bx) - \bm\alpha(t,\bm y)|^2
    \\
    &+ \int_{\R}\;\sum_{j\in\D}\; |\bm\gamma^{(j)}(t,\bx, z_j)-\bm\gamma^{(j)}(t,\bm y, z_j)|^2\nu_j(\diff z_j) \leq C_2\,|\bx - \bm y|^2\,.
\end{split}
\end{equation*}
\end{assumption}

As a consequence of Assumption \ref{assumption: lip and linear growth} and by Theorem 1.19 in \cite{oksendal2019applied}, there exists a unique càdlàg adapted process starting at $\bX_0 = \bx_0\in\R^n$ that satisfies the SDE in \eqref{eq: X dynamics ndim}; we refer to that process as this unique càdlàg solution  $(\bX_t)_\tT$. Moreover, it holds for all $\tT$ that
\begin{equation*}
    \E\left[ |\bX_t|^2\right]<\infty\,.
\end{equation*}

We use the Kullback-Leibler (KL) divergence also called relative entropy to quantify the distance between probability measures. Recall that the KL-divergence of a probability measure $\Q$ with respect to $\P$ is given by
\begin{equation*}
D_{KL} \left(\Q ~||~ \P \right)
=
\begin{cases}
\;\E\left[\,\dQP \,\log  \dQP  \,\right]
\quad &\text{if} \quad \Q \ll\P\\[1em]
\; \infty & \text{otherwise}\,,
\end{cases}
\end{equation*}
where we use the convention that $0 \log 0 = 0$. For a probability measures $\Q$ we write $\E^\Q[\cdot]$ when we consider the $\Q$-expectation and for notational simplicity set $\E[\cdot] := \E^\P[\cdot]$.

Now we are ready to formally introduce the optimisation problem which we will solve in the subsequent sections.
\begin{optimisation}
For functions $f_j,g_i:\R^n\to\R$ and constants $c_j, d_i\in\R$,  with $j \in \mR_1:=\{ 1,2,\dots,r_1\}$, $i \in \mR_2:=\{ 1,2,\dots,r_2\}$, we consider the optimisation problem\begin{equation}\tag{$P$}\label{opt}
\begin{aligned}
    \inf_{\Q \ll \P} D_{KL} \left(\Q ~||~ \P \right)
    \quad \text{subject to}
    \quad
    & \E^{\Q}\left[f_j(\bX_T)\right]=c_j\,, \;\;\forall\;j\in\mR_1\,, \quad \text{and}
    \\
    & 
    \E^{\Q}\left[\int_0^T g_i(\bX_s)\,\diff s\right]=d_i\,,\;\; \forall \; i\in\mR_2,
\end{aligned}
\end{equation}
where the infimum is taken over probability measures that are absolutely continuous with respect to $\P$. 
\end{optimisation}
For $j \in \mR_1$ and $i \in \mR_2$, we call the equations $\E^{\Q}\left[f_j(\bX_T)\right]=c_j$ and $\E^{\Q}\left[\int_0^T g_i(\bX_s)\,\diff s\right]=d_i$ constraints, and $f_j$ and $g_i$ constraint functions.

Before solving the optimisation problem \eqref{opt} we study the following closely related problem. Specifically, we consider optimisation problem \eqref{opt} however seek only over a subset of equivalent probability measures -- the set of equivalent probability measures characterised by Doléans-Dade exponentials. For this, we define the following sets of stochastic processes:
\begin{align*}
    \mcP_2\left([0,T]\right) 
    &:= \Big\{ \blambda ~\Big| ~
     \blambda: = (\blambda_t)_\tT \,\text{ is }\,\R^m\text{-valued }\F \text{-adapted and }
    \\[0.5em]
    & \qquad \qquad
    \E\left[\int_0^T |\blambda_t|^2\, \diff t \right] <\infty \Big\}\,,
\end{align*}
and
\begin{align*}
    \mcP_2([0,T]\times \R^l;\bm\nu) 
    &:= \Big\{\bh ~\Big|~ \bh: = (\bh_t(\bm z))_\tT \,\text{ is }\,\R^l \text{-valued, } \text{predictable, }
    \\[0.5em]
    & \qquad \qquad
    \bh_t(\bm z) = (h^1_t(z_1),\dots, h^l_t(z_l))\,, \quad \bh_t( \bm z)\leq 1\,, \quad \text{and} 
    \\[0.5em]
    &  \qquad \qquad
    \E\left[\int_0^T \int_{\R^l}  \left( \bh_t(\bm z)\odot\bh_t(\bm z)\right) \,\bnu(\diff \bm z) \diff t \right] <\infty \Big\}\,.
\end{align*}
Here $\odot$ stands for the Hadamard product for vectors, which is defined for $\bm x, \bm y \in \R^n$ by $\bm x \odot \bm y = (x_1\,y_1, \dots, x_n\,y_n)$, and the inequality $\bh_t( \bm z)\leq 1$  is to be understood componentwise.
For $\blambda \in\mcP_2\left([0,T]\right)$ and $\bh \in \mcP_2\left([0,T]\times \R^l;\nu\right)$, we define the process $Z^{\blambda,\bh}=(Z^{\blambda,\bh}_t)_{\tT}$, given for $\tT$ by
\begin{align}
\begin{split}
    Z^{\blambda,\bh}_t :=\exp\bigg( &- \int_0^t \blambda_s\,\diff \bW_s - \tfrac12 \,\int_0^t|\blambda_s|^2\,\diff s \\
    & +  \int_0^t \int_{\R^l} \blog(\bm 1-\bh_s(\bm z))\,\btmu(\diff s,\diff \bm z) \\
    & + \int_0^t \int_{\R^l} \{\blog(\bm 1-\bh_s(\bm z)) +\bh_s(\bm z)\}\,\bnu(\diff \bm z)\,\diff s\,\bigg)\,,\label{eq: Z_t for ndim case vect notation}
\end{split}
\end{align}
where $\blog(\bm 1- \bh_s(\bm z)):= \left(\,\log(1-h^1_s(z_1)),\dots, \log(1-h^l_s(z_l))\,\right)$. This process is a Dol\'eans-Dade exponential and with, e.g., the Novikov assumption, defines a Radon-Nikodym (RN) derivative. We recall the Novikov's condition on $ Z^{\blambda,\bh}$ which is
\begin{align*}
    \E\bigg[\exp\bigg( &\tfrac12 \,\int_0^T |\blambda_t|^2\,\diff t + \int_0^T\int_{\R^l} \bh_t(\bm z)\odot \bh_t(\bm z)\,\bmu(\diff t,\diff \bm z)\,\bigg)\bigg]<\infty\,,
\end{align*}
and which establishes sufficient conditions on $\blambda $ and $\bh $ such that $\E[Z^{\blambda,\bh}_T]=1$ and $(Z^{\blambda,\bh}_t)_{0\leq t \leq T}$ is a martingale. Thus, the measure $\Qlh$  characterised by the RN-derivative
$$\frac{\diff \Qlh }{\diff \P} = Z^{\blambda,\bh}_T$$
is a probability measure that is absolutely continuous with respect to $\P$; see Theorem 1.36 in \cite{oksendal2019applied}. We denote the probability measure $\Qlh$ by subscripts to indicates that it arises from an $\R^m$-valued process $\blambda$ and an $\R^l$-valued random field $\bh$.

With the above definitions we are ready to introduce a subset of absolutely continuous probability measures with respect to $\P$ given in \eqref{opt}, that are characterised by RN-densities $Z^{\blambda,\bh}$ with $\blambda \in\mcP_2\left([0,T]\right)$ and $\bh \in \mcP_2\left([0,T]\times \R^l;\nu\right)$
\begin{equation}\label{eq:Q-set}
\begin{split}
\mcQ :=
\Bigg\{\mathbb Q_{\blambda,\bh} 
&~ \Big|~\; \diff \Qlh = Z^{\blambda,\bh}_T\,\diff \P 
 \;\text{ s.t. }\; 
\E\left[Z^{\blambda,\bh}_T\right]=1,\,\; \E\left[\abs{Z^{\blambda,\bh}_T\,\log Z^{\blambda,\bh}_T}\right]<\infty\,,
\\
&\E^{\Qlh}\left[\abs{f_j(\bX_T)}\right]<\infty\,, \text{ }\,\forall\;j \in\mR_1\,,\quad \E^{\Qlh}\left[\int_0^T \abs{g_i(\bX_s)}\,\diff s\right]<\infty\,, \text{ }\,\forall\;i \in\mR_2\\
&\qquad\E^\Qlh\left[\sup_{t\in[0,T]} \abs{\bX_t}^2\right]<\infty\,,
 \text{ and } \blambda \in \mcP_2\left([0,T]\right)\,, \;
\bh \in \mcP_2\left([0,T]\times \R^l;\nu\right)
\Bigg\}\,.
\end{split}
\end{equation}
Note that we do not assume Novikov's condition in \eqref{eq:Q-set}.

Using the above class of equivalent probability measures, we consider the following optimisation problem, which is optimisation problem \eqref{opt} but where we seek over the subset of probability measures $\mcQ$.
\begin{optimisation}
For functions $f_j,g_i:\R^n\to\R$ and constants $c_j, d_i\in\R$, with $j \in \mR_1$, $i \in \mR_2$, we consider the optimisation problem
\begin{equation}\label{opt-Z-lambda-h}
 \tag{$P^\prime$}
\begin{aligned}
    \inf_{\Qlh\in\mcQ} \E\left[Z^{\blambda,\bh}_T\,\log Z^{\blambda,\bh}_T\right]
    \quad \text{subject to}
    \quad
    & \E^{\Qlh}\left[f_j(\bX_T)\right]=c_j\,, \text{ }\,\forall\;j\in\mR_1\,, \quad \text{and}
    \\
    &
    \E^{\Qlh}\left[\int_0^T g_i(\bX_s)\,\diff s\right]=d_i\,, \text{ }\,\forall\;i\in\mR_2\,.
\end{aligned}
\end{equation}
\end{optimisation}

For $\Qlh\in\mcQ$, and as a consequence of  Girsanov's Theorem, $\bWlambda$ defined by
\begin{equation*}
    \bWlambda_t := \int_0^t \blambda^{\intercal}_s\,\diff s + \bW_t
\end{equation*}
is an $m$-dimensional $\Qlh$-Brownian motion and the $\Qlh$-compensator of
$\bmu $ is
\begin{equation*}
 \bnu^{\bh}(\diff t, \diff \bm z)
 :=
 (\bm 1-\bh_t( \bm z))^{\intercal}\odot\,\bnu(\diff \bm z)\,\diff t   \,.
\end{equation*}
For notational simplicity  we write $\bWlambda$ and $\bnu^{\bh}$ as they only explicitly depend on $\blambda$ and $\bh$, respectively.

Using the above results, the KL-divergence from $\Qlh$ to $\P$ becomes
\begin{equation*}
    \begin{split}
        &D_{KL}\left(\Qlh ~||~ \P\right)
        \\
        & \quad=\E^{\Qlh}\left[\tfrac{1}{2}\int_0^T |\blambda_t|^2\,\diff t  \right.
        + \left.\int_0^T\int_{\R^l} \big[\blog(\bm 1-\bh_t(\bm z))\odot(\bm 1-\bh_t(\bm z))  + \bh_t(\bm z) \big]\,\bnu(\diff \bm z) \diff t  \right].
    \end{split}
\end{equation*}

Next, we discuss assumptions needed for the existence and uniqueness of the Lagrangian associated with optimisation problem \eqref{opt-Z-lambda-h}, which we introduce in the next section. For this we first define the moment generating function (mgf) and the cumulant generating function (cgf) for random vectors. For a random vector $\bY = (Y_1, \ldots, Y_k)$, $k\in \R$, we define the set 
\begin{equation*}
    D_{\bY}:=\left\{\ba\in\R^{k} ~\big|~ \E\left[\exp(\ba\,\cdot \,\bY)\right]<\infty\right\}^\circ\,,
\end{equation*}
where $\{\}^\circ$ denotes the interior of a set. We note that $D_{\bY}$ is the interior of a convex set.
If $D_{\bY} \neq \emptyset$, then the mgf $M_{\bY}$ and cgf $K_{\bY}$ of $\bY$ at $\ba \in D_{\bY}$ exist and are respectively given by
\begin{equation*}
    M_{\bY} (\ba) 
    = \E[\,\exp(\ba\,\cdot \bY)\,]
    \qquad \text{and}\qquad
    K_{\bY}(\ba) 
    = \log M_{\bY} (\ba)\,.
\end{equation*}

\begin{assumption}\label{assumption: existence! Lagrange}
Let $\mfX$ denote the $(r_1+r_2)$-dimensional random vector given by
\begin{equation*}
    \mfX 
    :=  \left(\,\bm f(\bX_T)-\bm c,\,\int_0^T\bg(\bX_s)\,\diff s-\bd\,\right)\,,
\end{equation*}
where $\bm f(\bX_T):= (f_1(\bX_T),\dots,f_{r_1}(\bX_T))$, $\bg(\bX_T):= (g_1(\bX_T),\dots,g_{r_2}(\bX_T))$, and constants $\bm c := (c_1,\dots,c_{r_1})$, and  $\bm d := (d_1,\dots,d_{r_2})$. Here the integral in $\int_0^T\bg(\bX_s)\diff s$ is understood to be applied componentwise.
We assume that $D_{\mfX}\neq \emptyset$ and that there exists $\ba$ such that 
\begin{equation}\label{eq: condition lagrange}
    \nabla_{\ba} K_{\mfX}(-\ba) = \bm 0\,.
\end{equation}
\end{assumption}

In the next sections we first solve optimisation problem \eqref{opt-Z-lambda-h} and then show that its solution, if it exists, is the also the solution to optimisation problem \eqref{opt}. To solve optimisation problem \eqref{opt-Z-lambda-h} we next proceed by presenting a formal derivation of a candidate solution and a verification theorem.

\subsection{Candidate Solution and Verification}\label{sec:verification}
We proceed with a formal derivation of a candidate for the value function associated with the constrained optimisation problem \eqref{opt-Z-lambda-h}. After, we provide a verification theorem that allows us to conclude that the candidate solution is indeed the value function. 

Let $(\eone, \etwo)\in D_{-\mfX}$ with $\eone:=(\eta_1,\dots,\eta_{r_1}) \in \R^{r_1}$ and $\etwo:=(\eta_{r_1+ 1},\dots,\eta_{r_1 + r_2}) \in \R^{r_2}$, then the Lagrangian of the constrained problem \eqref{opt-Z-lambda-h} with Lagrange multipliers $\bm\eta_1$ and $\bm \eta_{2}$ is given by
\begin{equation*}
    L^{\blambda,\bh} := \E^{\Qlh}\left[\log Z^{\blambda,\bh}_T + \eone \cdot \left(\bm f(\bX_T)-\bm c\right) + \etwo \cdot \left(\int_0^T \bg(\bX_s)\diff s - \bd\right)\right]\,,
    \label{eqn:multi Functional-Lagrange}
\end{equation*}
where $\,\cdot\,$ denotes the dot product.
We define for a fixed control pair $(\blambda,\bh)$, the value $J^{\blambda,\bh}:[0,T]\times\R^n \to \R$ associated with the Lagrangian $L^{\blambda,\bh}$ by 
\begin{equation}
    \label{eqn:multi Functional-Lagrange-t}
    \begin{split}
        J^{\blambda,\bh}(t,\bx) :=  \E^{\Qlh}_{t,\bx}\Bigg[ \,&  \tfrac{1}{2}\int_t^T |\blambda_t|^2\,\diff t 
        \\
        &  + \int_t^T\int_{\R^l} \big[\blog(\bm 1-\bh_t(\bm z))\odot(\bm 1-\bh_t(\bm z))  + \bh_t(\bm z) \big]\,\bnu(\diff \bm z) \diff t  
        \\
        &  + \eone \cdot\left(\bm f(\bX_T)-\bm c\right) + \etwo\cdot\left(\int_t^T \bg(\bX_s)\diff s - \bd\right)\,\Bigg]\,,
    \end{split}
\end{equation}
where $\E^{\Qlh}_{t,\bx}[\cdot]$ denotes the $\Qlh$-expectation conditioned on the event $\bX_{t}=\bx$. Observe that the expectation in \eqref{eqn:multi Functional-Lagrange-t} is finite because of the definition of $\mcQ$ -- recall that $D_{KL}(\Qlh~ || ~ \P) = \E\left[{Z^{\blambda,\bh}_T\,\log Z^{\blambda,\bh}_T}\right]<\infty$.
We further define the optimal value function, which we often just refer to as the value function, by
\begin{equation}
    J(t,\bx) := \inf_{\substack{\blambda, \bh, \; s.t. \\ 
    \Qlh\in\mcQ
    }} J^{\blambda,\bh}(t,\bx)\,.
    \label{eqn:value-function-def}
\end{equation}
For the purposes of the formal derivation we assume that the infimum in \eqref{eqn:value-function-def} is finite.
We observe that as a consequence of the dynamic programming principle and It\^o's formula -- under the assumption that $J\in\mathcal{C}^{1,2}([0,T)\times \R^n; \R)\cap \mathcal{C}^{0}([0,T]\times \R^n; \R)$ -- we have the following dynamic programming equation (DPE) 
\begin{multline}\label{eqn:multi HJB}
    \partial_t J(t,\bx)+ \inf_{\blambda, \bh} \Bigg\{ \Lg^{\blambda,\bh} J(t,\bx) + \tfrac12 |\blambda|^2 
    \\
    \qquad \qquad  \quad + \int_{\R^l} \big[\blog(\bm 1-\bh_t(\bm z))\odot(\bm 1-\bh_t(\bm z))  + \bh_t(\bm z) \big]\,\bnu(\diff \bm z) +\etwo \cdot \bg(\bx)\Bigg\} =0\,, \\
    J(T,\bx) = \eone \cdot \left(\bm f(\bx)-\bm c\right) - \etwo\cdot\,\bd \,,
\end{multline}
where the linear operator $\Lg^{\blambda,\bh}$ is the $\Qlh$-generator of $\bX$, and acts on functions as follows
\begin{align*}
\begin{split}
       \Lg^{\lambda,h}J(t,\bx)  &=\left(\bm \alpha(t,\bx) - \bm\sigma(t,\bx)\,\blambda^{\intercal}\right)\cdot\,\nabla_{\bx}J + \tfrac{1}{2}\,\Tr\left(\bm\sigma(t,\bx)\,\bm\sigma(t,\bx)^{\intercal}\,\nabla^2_{\bx} J\right)
       \\
    &\quad + \int_{\R^l} \bm\Delta_{\bm z}J(t,\bx)\, \bnu^{\bh}(\diff \bm z)  - \int_{\R^l} (\nabla_{\bx}J)^{\intercal}\,\bm\gamma(t,\bx,\bm z)\,\bnu(\diff \bm z)\,,
\end{split}
\end{align*}
where $\bm\Delta_{\bm z}J(t,\bx)$ is defined in \eqref{eq:def-Delta-z}.
The specific form of the DPE follows from writing \eqref{eq: X dynamics ndim} in terms of $\bW^\blambda$ and $\btmu^{\bh} = \bmu - \bnu^{\bh}$, so that
\begin{multline*}
    \diff \bX_t 
    = 
    \left(\bm\alpha(t,\bX_t) - \bm\sigma(t,\bX_t)\,\blambda_t^{\intercal}   - \int_{\R^l} \bm\gamma(t,\bX_{t^-},\bm z) \left[\bh^{\intercal}_t(\bm z) \odot \bnu(\diff \bm z)\right] \right)\,\diff t
    \\
     +\bm\sigma(t,\bX_t)\,\diff \bWlambda_t + \int_{\R^l}\bm \gamma(t,\bX_{t^-},\bm z)\,\btmu^{\bh}(\diff t, \diff \bm z)\,.
\end{multline*}
The measurable global minimisers $\blambda^\dagger$ and $\bh^\dagger$ (in feedback form) of the infimum in \eqref{eqn:multi HJB}  are given by
\begin{align*}
    \blambda^\dagger(t,\bx) &= \big(\nabla_{\bx} J(t,\bx)\big)\, \bm\sigma(t, \bx)\,,
    \\
    \bh^\dagger(t,\bx,\bm z)&= \bm 1 - \bm{e}^{-\bm\Delta_{\bm z}J(t,\bx)}\,,
\end{align*}
where $ \bm 1 - \bm{e}^{-\bm\Delta_{\bm z}J(t,\bx)}$ stands for $\left(1-e^{-\Delta^{1}_{z_1} J(t,\bx)},\dots,1-e^{-\Delta^{l}_{z_l} J(t,\bx)}\right)$.

It follows immediately that $\bh^\dagger$ is componentwise bounded from above by unity. Inserting the optimal controls $\blambda^\dagger$ and $\bh^\dagger$ in feedback form back into the DPE \eqref{eqn:multi HJB} (omitting the arguments $(t,\bx)$ when possible)  we obtain that
\begin{multline}
\label{eq: multi PDE vnot1}
    \partial_t J - \tfrac{1}{2}\,|\nabla_{\bx}J\,\bm\sigma|^2 
    + \bm\alpha\cdot \nabla_{\bx}J - \int_{\R^l} (\nabla_{\bx}J)^{\intercal}\,\bm\gamma(t,\bx,\bm z)\,\bnu(\diff \bm z)
    \\
    \quad+ \tfrac{1}{2}\,\Tr\left(\bm\sigma\,\bm\sigma^{\intercal}\,\nabla^2_{\bx}J\right)  +\int_{\R^l} \left( \bm 1 - \bm{e}^{-\bm\Delta_{\bm z}J} \right) \bnu(\diff \bm z)  +\etwo \cdot \bg = 0\,
\end{multline}
together with the terminal condition $J(T,\bx) = \sum_{j=1}^{r_1} \eta_j\,(f_j(\bx)-c_j) - \sum_{i =r_1+1}^{r_2}\eta_i\,d_i$.
We observe that \eqref{eq: multi PDE vnot1} can be written as
\begin{multline}
    \partial_t J + \Lg^c J - \tfrac{1}{2}\,|\nabla_{\bx}J\,\bm\sigma|^2 
    -
    \int_{\R^l} (\nabla_{\bx}J)^{\intercal}\,\bm\gamma(t,\bx,\bm z)\,\bnu(\diff \bm z)
     \\
     +
     \int_{\R^l} \left( \bm 1 - \bm{e}^{-\bm\Delta_{\bm z}J} \right) \bnu(\diff \bm z)  +
     \etwo \cdot \bg =0\,,\label{eq: multi PDE vnot2}
\end{multline}
where the linear operator $\Lg^c$ is the $\P$-generator of the continuous part\footnote{The continuous part of $\bX$ is defined by $\bX^c_t:=\bX_t- \sum_{0\le s\le t} \Delta \bX_s$, $\Delta\bX_t:=\bX_t-\bX_{t^-}$, where $\bX_{t^-}:=\lim_{s\uparrow t}\bX_s$. } of $\bX$ and acts on functions as follows
\begin{equation*}
    \Lg^c J = \bm\alpha\cdot \nabla_{\bx}J + \tfrac{1}{2}\,\Tr\left(\bm\sigma\,\bm\sigma^{\intercal}\,\nabla^2_{\bx}J\right)\,.
\end{equation*}
Next, we construct a candidate of the solution to \eqref{eq: multi PDE vnot2} by introducing the change of variables $J(t,\bx)=-\log \omega(t,\bx)$. Hence,
\begin{equation*}
    \partial_tJ =- \frac{\partial_t\omega}{\omega},
\quad
    \nabla_{\bx}J = -\frac{1}{\omega}\,\nabla_{\bx}\omega,
\quad
    \nabla^2_{\bx}J = -\frac{1}{\omega}\,\nabla^2_{\bx}\omega +\frac{1}{\omega^2}(\nabla_{\bx}\omega)^{\intercal}\,\nabla_{\bx}\omega\,,
\qquad
\end{equation*}
and furthermore $\Delta^{j}_{z_j} J(t,\bx) =  \log \left(w(t,\bx)/w(t,\bx+\bm\gamma^{(j)}(t,\bx, z_j))\right)$. Equation \eqref{eq: multi PDE vnot2} thus becomes
\begin{multline}
    -\frac{1}{\omega}\left\{ \partial_t \omega + \bm\alpha\cdot\nabla_{\bx}\omega + \tfrac{1}{2}\Tr\left(\bm\sigma\,\bm\sigma^{\intercal}\,\nabla^2_{\bx}\omega\right)- \int_{\R^l} (\nabla_{\bx}\omega)^{\intercal}\,\bm\gamma\,\bnu(\diff \bm z) + \int_{\R^l} \bm\Delta_{\bm z}\omega\, \bnu(\diff \bm z) \right\}
    \\
    + \frac{1}{2\,\omega^2}\,\Tr\left(\bm\sigma\,\bm\sigma^{\intercal}\,(\nabla_{\bx}\omega)^{\intercal}\,\nabla_{\bx}\omega\right) - \frac{1}{2\,\omega^2}\left|\nabla_{\bx}\omega\,\bm\sigma\right|^2 + \etwo \cdot \bg = 0\,,\label{eq: PDE transformation 1}
\end{multline}
with $\omega(T,\bx) = \exp\left(-\eone \cdot \left(\bm f(\bx)-\bm c\right) + \etwo\cdot\bd\right)$.
Multiplying \eqref{eq: PDE transformation 1} by $-\omega(t,\bx)$, we have that
\begin{align*}
    &  \partial_t \omega + \bm\alpha\cdot\nabla_{\bx}\omega + \tfrac{1}{2}\Tr\left(\bm\sigma\,\bm\sigma^{\intercal}\,\nabla^2_{\bx}\omega\right)- \int_{\R^l} (\nabla_{\bx}\omega)^{\intercal}\,\bm\gamma\,\bnu(\diff \bm z) + \int_{\R^l} \bm\Delta_{\bm z}\omega\, \bnu(\diff \bm z) - \etwo \cdot \bg \,\omega = 0\,,
\end{align*}
where we use the fact that
\begin{equation*}
    \frac{1}{2\,\omega^2}\,\Tr\left(\bm\sigma\,\bm\sigma^{\intercal}\,(\nabla_{\bx}\omega)^{\intercal}\,\nabla_{\bx}\omega\right) - \frac{1}{2\,\omega^2}\left|\nabla_{\bx}\omega\,\bm\sigma\right|^2  = 0\,.
\end{equation*}
By the Feynman-Kac representation, we conclude that $\omega(t,\bx)$ can be written as
\begin{equation*}
    \omega(t,\bx) = \E_{t,\bx}\left[ \exp\left( -\eone \cdot \left(\bm f(\bX_T)-\bm c\right) 
    -
    \etwo \cdot \left(\int_t^T \bg(\bX_u)\diff u - \bd\right)\right)  \right]\,.
\end{equation*}

The above formal calculations provide  the following candidate solution for the value function.
\begin{proposition}
\label{prop:candidate}
A candidate solution to the value function \eqref{eqn:value-function-def}  is given by
\begin{equation*}
    J(t,\bx) = -\log \,\E_{t,\bx}\left[ \exp\left( -\eone \cdot \left(\bm f(\bX_T)-\bm c\right) 
-
    \etwo \cdot\left(\int_t^T \bg(\bX_u)\diff u - \bd\right)\right)  \right]\,,
\end{equation*}
with the optimal Markovian controls given by
\begin{align*}
    \blambda_t := \nabla_{\bx} J(t,\bX_t)\,\bm\sigma(t,\bX_t)
    \qquad \text{and} \qquad
    \bh_t(\bm z) := \bm 1 - \bm{e}^{-\bm\Delta_{\bm z}J(t,\bX_{t^-})} \,.
\end{align*}
\end{proposition}

Next, we prove that, under certain conditions and for fixed Lagrange multipliers, this candidate solution does indeed coincide with the value function.
\begin{theorem}[Verification]\label{thm: verification}
Under Assumption \ref{assumption: lip and linear growth}, let $D_{-\mfX} \neq \emptyset$ and $(\eone,\etwo)\in D_{-\mfX}$. Define
\begin{equation*}
    J^\dagger(t,\bx) := -\log\omega^\dagger(t,\bx)\,,
\end{equation*}
where 
\begin{equation}
\label{eq:omega-star}
    \omega^\dagger(t,\bx) = \E_{t,\bx}\left[ \exp\left( - \eone \cdot \left(\bm f(\bX_T)-\bm c\right) - \etwo\cdot \left(\int_t^T \bg(\bX_u)\diff u - \bd\right)\right)  \right]
\end{equation}
and suppose that
$J^\dagger \in\mathcal{C}^{1,2}([0,T)\times \R^n; \R)\cap \mathcal{C}^{0}([0,T]\times \R^n; \R)$ with $J^\dagger$ having at most quadratic growth, i.e., there is $C_3\in\R^+$ such that $\abs{J(t,\bx)}\leq C_3(1+\abs{\bx}^2)$. Let
\begin{subequations}
\label{eqn:lambda-h-dagger}
\begin{align}
    \blambda^\dagger_t &:= 
    -\frac{\nabla_{\bx} \omega^\dagger(t,\bX_t) \,}{\omega^\dagger(t,\bX_t)}\,\bm\sigma(t,\bX_t)
    \qquad \text{and} \\
    \bh^\dagger_t(\bm z) &:= 
    - \frac{\bm\Delta_{\bm z}\omega^\dagger(t,\bX_{t^-})}{\omega^\dagger(t,\bX_{t^-})}\,,
\end{align}
\end{subequations}
and assume that
\begin{align}\label{eq: sufficient cond admissible}
    \E\bigg[\exp\bigg( &\tfrac12 \,\int_0^T |\blambda^\dagger_s|^2\,\diff s + \int_0^T\int_{\R^l} \bh^\dagger_t(\bm z)\odot \bh^\dagger_t(\bm z)\,\bmu(\diff t,\diff \bm z)\,\bigg)\bigg]<\infty\,,
\end{align}
\begin{equation}\label{eq: sufficient cond admissible I}
\E^{\Qlh}\left[\abs{f_j(\bX_T)}\right]<\infty\,, \text{ }\,\forall\;j \in\mR_1\,,\quad \E^{\Qlh}\left[\int_0^T \abs{g_i(\bX_s)}\,\diff s\right]<\infty\,, \text{ }\,\forall\;i \in\mR_2\,,    
\end{equation}
and
\begin{equation}\label{eq: sufficient cond admissible II}
\E^{\Q^{\blambda^\dagger,\bh^\dagger}}\left[\sup_{t\in[0,T]} \abs{\bX_t}^2\right]<\infty \,.    
\end{equation}
Then, $\blambda^\dagger$ and $\bh^\dagger$ are admissible controls and $J^\dagger = J$. 
\end{theorem}
\begin{proof}
Note that \eqref{eq: sufficient cond admissible} and \eqref{eq: sufficient cond admissible II} are sufficient to guarantee that $\blambda^\dagger$ and $\bh^\dagger$ are admissible and that they induce a measure $\Q_{\blambda^\dagger,\bh^\dagger}$ that is well-defined -- this is a consequence of Novikov's condition and the definition of $\mcQ$. Next, we observe that 
\begin{align*}
   & \partial_t J^\dagger(t,\bx)+ \inf_{\blambda, \bh} \Bigg\{ \Lg^{\blambda,\bh} J^\dagger(t,\bx) + \tfrac12 |\blambda|^2 
    \\
    &    \qquad \qquad  \quad + \int_{\R^l} \big[\blog(\bm 1-\bh_t(\bm z))\odot(\bm 1-\bh_t(\bm z))  + \bh_t(\bm z) \big]\,\bnu(\diff \bm z) +\etwo\cdot \bg(\bx)\Bigg\} 
\\
 & = \partial_t J^\dagger(t,\bx)+  \Lg^{\blambda^\dagger,\bh^\dagger} J^\dagger(t,\bx) + \tfrac12 |\blambda^\dagger|^2 
    \\
  &  \qquad \qquad  \quad + \int_{\R^l} \big[\blog(\bm 1-\bh^\dagger_t(\bm z))\odot(\bm 1-\bh^\dagger_t(\bm z))  + \bh^\dagger_t(\bm z) \big]\,\bnu(\diff \bm z)  +\etwo \cdot \bg(\bx) \\
  &= 0\,.
\end{align*}
Then, for arbitrary $(\blambda,\bh)$ s.t. $\Qlh \in\mcQ$, $(t,\bx)\in[0,T)\times\R^n$, $s\in[t,T)$ and a stopping time $\tau_n = \inf\{u\geq t\,:\, |\bX_u|>n\}$, for $n\in\Z_+, n<\infty$, we use Dynkin's formula to obtain
\begin{align*}
   \E_{t,\bx}^{\Qlh} \left[J^\dagger(s\wedge \tau_n, \bX_{s\wedge \tau_n})\right] 
   &= 
   J^\dagger(t,\bx) + \E_{t,\bx}^{\Qlh} \left[\int_t^{s\wedge\tau_n} \partial_t J^\dagger(u,\bX_u) + \mathcal{L}^{\blambda,\bh} J^\dagger(u,\bX_u) \,\diff u\right]\,,
\end{align*}
and as
\begin{multline*}
   \partial_t J^\dagger(t,\bx)+ \Lg^{\blambda,\bh} J^\dagger(t,\bx) + \tfrac12 |\blambda|^2 
    \\
 + \int_{\R^l} \big[\blog(\bm 1-\bh_t(\bm z))\odot(\bm 1-\bh_t(\bm z))  + \bh_t(\bm z) \big]\,\bnu(\diff \bm z) +\etwo \cdot \bg(\bx)\geq 0,
\end{multline*}
we conclude that
\begin{align*}
   \E_{t,\bx}^{\Qlh} 
   &
   \left[  J^\dagger(s\wedge \tau_n, \bX_{s\wedge \tau_n})\right] 
   \\
   &\geq
   J^\dagger(t,\bx) - \E_{t,\bx}^{\Qlh} \left[\int_t^{s\wedge\tau_n} \tfrac12 |\blambda_u|^2 +\etwo \cdot \bg(\bX_u)\diff u \right] 
   \\
   &- \E_{t,\bx}^{\Qlh} \left[\int_t^{s\wedge \tau_n}\int_{\R^l} \big[\blog(\bm 1-\bh_u(\bm z))\odot(\bm 1-\bh_u(\bm z))  + \bh_u(\bm z) \big]\,\bnu(\diff \bm z)  \,\diff u\right]\,.
\end{align*}
Since $\Qlh\in\mcQ$ we have that 
\begin{align*}
    &
    \abs{\E_{t,\bx}^{\Qlh} \left[\int_t^{s\wedge \tau_n}\tfrac12 |\blambda_u|^2 
    +
    \int_{\R^l} \big[\blog(\bm 1-\bh_u(\bm z))\odot(\bm 1-\bh_u(\bm z))  
    + \bh_u(\bm z) \big]\,\bnu(\diff \bm z)  \,\diff u\right]} 
    \\
    &\quad \leq \E_{t,\bx}^{\Qlh} \left[\int_0^{T}\tfrac12 |\blambda_u|^2 +\int_{\R^l}  \big[\blog(\bm 1-\bh_u(\bm z))\odot(\bm 1-\bh_u(\bm z))  + \bh_u(\bm z) \big] \,\bnu(\diff \bm z)  \,\diff u\right] < \infty\,,
\end{align*}
where we used that $\log(1-y) (1-y) + y \ge 0$ for $y \le 1$.
Similarly, by \eqref{eq: sufficient cond admissible I} 
\begin{align*}
    &\abs{\E_{t,\bx}^{\Qlh} \left[\int_t^{s\wedge\tau_n}  \etwo\cdot \bg(\bX_u)\diff u \right]} 
    \leq 
    \abs{\etwo} \cdot\, \E_{t,\bx}^{\Qlh} \left[\int_0^{T} \,\abs{\bg(\bX_u)}\diff u \right] <\infty\,.
\end{align*}
Finally, using the quadratic growth condition imposed on $J^\dagger$, we have 
\begin{equation*}
    \abs{J^\dagger(s\wedge \tau_n, \bX_{s\wedge \tau_n})} \leq C_3\left(1+\sup_{u\in[t,T]}\abs{\bX_u}^2\right)\,,
\end{equation*}
and the right hand side of the inequality is integrable with respect to $\Qlh$ because $\Qlh\in\mcQ$.
Thus, as a consequence of the dominated convergence theorem, we can take the limit when $n\to \infty$ to obtain for all $s \in[t, T)$
\begin{align*}
   \E_{t,\bx}^{\Qlh} \left[J^\dagger(s, \bX_{s})\right] &\geq J^\dagger(t,\bx) - \E_{t,\bx}^{\Qlh} \left[\int_t^{s} \tfrac12 |\blambda_u|^2 +\etwo \cdot \bg(\bX_u)\diff u \right] 
   \\
   &- \E_{t,\bx}^{\Qlh} \left[\int_t^{s}\int_{\R^l} \big[\blog(\bm 1-\bh_u(\bm z))\odot(\bm 1-\bh_u(\bm z))  + \bh_u(\bm z) \big]\,\bnu(\diff \bm z)  \,\diff u\right]\,,
\end{align*}
and by continuity of $J^\dagger$, as we send $s\nearrow T$, we obtain 
\begin{align*}
   &\E_{t,\bx}^{\Qlh} \Big[\eone \cdot\left(\bm f(\bX_T)-\bm c\right) - \etwo\cdot \bd\Big] 
   \geq 
   \\
   & \qquad J^\dagger(t,\bx) - \E_{t,\bx}^{\Qlh} \left[\int_t^{T} \tfrac12 |\blambda_u|^2 +\etwo\cdot\bg(\bX_u)\diff u \right] 
   \\
   &\qquad - \E_{t,\bx}^{\Qlh} \left[\int_t^{T}\int_{\R^l} \big[\blog(\bm 1-\bh_u(\bm z))\odot(\bm 1-\bh_u(\bm z))  + \bh_u(\bm z) \big]\,\bnu(\diff \bm z)  \,\diff u\right]\,.
\end{align*}
After rearranging the above equation we have that $J^\dagger \le J^{\blambda,\bh}$, and as a consequence of the arbitrariness of $\blambda$ and $\bh$, we obtain $J^\dagger \leq J$.

Finally, using a similar localisation technique as above, this time with $(\blambda^\dagger,\bh^\dagger) $ and corresponding measure ${\Q^{\blambda^\dagger,\bh^\dagger}}\in\mcQ$, and since it holds by construction that
\begin{align*}
  &  \partial_t J^\dagger(t,\bx)+  \Lg^{\blambda^\dagger,\bh^\dagger} J^\dagger(t,\bx) + \tfrac12 |\blambda^\dagger|^2 
    \\
  &  \qquad \qquad  \quad + \int_{\R^l} \big[\blog(\bm 1-\bh^\dagger_t(\bm z))\odot(\bm 1-\bh^\dagger_t(\bm z))  + \bh^\dagger_t(\bm z) \big]\,\bnu(\diff \bm z)  +\etwo \cdot \bg(\bx) = 0\,,
\end{align*}
we have that 
\begin{align*}
    &\E_{t,\bx}^{\Q^{\blambda^\dagger,\bh^\dagger}} \big[\eone \cdot\left(\bm f(\bX_T)-\bm c\right) - \etwo\cdot \bd\big]
    \\
    & \qquad 
    = J^\dagger(t,\bx) - \E_{t,\bx}^{\Q^{\blambda^\dagger,\bh^\dagger}} \left[\int_t^{T} \tfrac12 |\blambda^\dagger_u|^2 +\etwo \cdot \bg(\bX_u)\diff u \right] 
    \\
    &\qquad\quad - \E_{t,\bx}^{\Q^{\blambda^\dagger,\bh^\dagger}} \left[\int_t^{T}\int_{\R^l} \big[\blog(\bm 1-\bh^\dagger_u(\bm z))\odot(\bm 1-\bh^\dagger_u(\bm z))  + \bh^\dagger_u(\bm z) \big]\,\bnu(\diff \bm z)  \,\diff u\right]\,.
\end{align*}
Therefore, after rearranging, we have that $J\leq J^\dagger$. Combining both inequalities we obtain that $J=J^\dagger$ which concludes the proof.
\end{proof}

We introduce the notation $\Q^\dagger:=\Q_{\blambda^\dagger,\bh^\dagger}$ to refer to the measure change induced by choosing $\blambda^\dagger,\bh^\dagger$ as in \eqref{eqn:lambda-h-dagger}.

\subsection{Representation of RN-density}\label{sec:RN-representation}

The next result shows that for $\Qlh\in \mcQ$ with carefully chosen $\blambda$ and $\bh$, the corresponding RN-density $Z^{\blambda,\bh}_T$ of $\Qlh$ has an alternative representation. This leads to a simple representation of the RN-density that characterises a solution to $\eqref{opt-Z-lambda-h}$, see Corollary \ref{cor:dQ-dP}.

\begin{theorem}[Representation of RN-density]\label{thm: form of optimal RN derivative}
Let Assumption \ref{assumption: lip and linear growth} be fulfilled. Let  $Z^{\blambda,\bh}$ be given in \eqref{eq: Z_t for ndim case vect notation} with $\blambda,\bh$ specifically chosen to be
\begin{align*}
       \blambda_t = -\frac{\nabla_{\bx} w(t,\bX_t) \,}{w(t,\bX_t)}\,\bm\sigma(t,\bX_t)\,\qquad \text{and}\qquad
    \bh_t(\bm z) = - \frac{\bm\Delta_{\bm z}w(t,\bX_{t^-})}{w(t,\bX_{t^-})}\,,
\end{align*}
where $w:[0,T]\times \R^n\to \R$ is
\begin{equation*}
    w(t,\bx) := \E_{t,\bx}\left[H\left(\bX_T\right) \,G\left(\int_t^T \ell(\bX_s)\,\diff s\right)\right]\,,
\end{equation*}
with $\ell \colon \R^n \to \R$ and for some $H:\R^n\to\R^+\setminus \{0\}$ such that $\E[H(X_T)]<\infty$, and $G:\R\to\R^+\setminus \{0\}$ is $\mathcal{C}^{1}$ and such that $G(a)\,G( b) = G(a+b)$ for $a,b\in\R$.\footnote{From this property, it follows that $G(0)=1$, and $G(-a) = 1/G(a)$. In fact, a simple calculation shows that  $G(x)$ is of the form $\exp(\kappa\,x)$ for some constant $\kappa$.} 
We assume that $H$ and $G$ are such that $w\in\mathcal{C}^{1,2}([0,T]\times \R^n)$ and \begin{align}\label{eq: novikov cond thm1.2}
  \E\bigg[\exp\bigg( &\tfrac12 \,\int_0^T |\blambda_s|^2\,\diff s + \int_0^T\int_{\R^l} \bh_t(\bm z)\odot \bh_t(\bm z)\,\bmu(\diff t,\diff \bm z)\,\bigg)\bigg]<\infty\,.
\end{align}
Then, we have, for all $\tT$, that
\begin{equation*}
    Z^{\blambda,\bh}_t = \frac{w(t,\bX_{t})\,G\left(\int_0^t \ell(\bX_s)\,\diff s\right)}{\E\left[w(T,\bX_T)\,G\left(\int_0^T \ell(\bX_s)\,\diff s\right)\right]}\,,
\end{equation*}
and, noting that $w(T,\bX_T) = H(\bX_T)\,G(0)$ and $G(0)=1$,
\begin{equation*}
    Z^{\blambda,\bh}_T = \frac{H(\bX_T)\,G\left(\int_0^T \ell(\bX_s)\,\diff s\right)}{\E\left[H(\bX_T)\,G\left(\int_0^T \ell(\bX_s)\,\diff s\right)\right]}\,.
\end{equation*}
\end{theorem}

\begin{proof}
For simplicity we drop the superscripts of $Z^{\blambda,\bh}$ and just write $Z$. 
By \eqref{eq: novikov cond thm1.2} we have that $(Z_t)_{t \in [0,T]}$ is a martingale and thus $Z_t = \E_t[Z_T]$ for $0\leq t \leq T$, where $\E_t[\cdot] := \E[\cdot | \F_t]$. As $Z_t$ is a stochastic exponential, from It\^o's lemma, we have that
\begin{equation}
\label{eqn:dZ-stoch-exp}
    \diff Z_t = \frac{Z_{t^-}}{w}\,\nabla_{\bx}w\,\bm\sigma\,\diff \bW_t  + \frac{Z_{t^-}}{w}\int_{\R^l} \bm\Delta_{\bm z}w(t,\bX_{t^-})\,\tilde{\bmu}(\diff t,\diff\bm z)\,.
\end{equation}
Next, we observe that the process \begin{equation*}
\tilde{w}_t := G\left(\int_0^t \ell(\bX_s)\diff s\right)\, w(t,\bX_{t}) =  \E_{t,\bX_{t}}\left[H\left(\bX_T\right) \,G\left(\int_0^T \ell(\bX_s)\diff s\right)\right]\,,
\end{equation*}
is a martingale. Given that $w\in\mathcal{C}^{1,2}([0,T)\times\R^n)$, it follows that $\tilde{w}$ satisfies the following SDE
\begin{equation*}
\begin{split}
    \diff \tilde{w}_t 
    =&
    \; G\left(\textstyle\int_0^t \ell(\bX_s)\,\diff s\right)\Bigg\{\,\partial_t w\,\diff t + (\nabla_{\bx}w)\cdot\bm\alpha\,\diff t + (\nabla_{\bx}w)\,\bm\sigma\,\diff \bW_t 
    \\
    & \qquad \qquad\qquad\qquad + \tfrac{1}{2}\,\Tr\left(\bm\sigma\,\bm\sigma^{\intercal}\,\nabla^2_{\bx}w\right)\,\diff t + \int_{\R^l} \bm\Delta_{\bm z}w\,{\bmu}(\diff t,\diff\bm z)\Bigg\}
    \\
    &   + w\,G'\left(\textstyle\int_0^t \ell(\bX_s)\,\diff s\right)\,\ell(\bX_t)\,\diff t
    \,,
\end{split}
\end{equation*}
and as $\tilde{w}$ is a martingale, we have the following identity
\begin{multline}
\label{eq: identity thm v2}
    G\left(\textstyle\int_0^t \ell(\bX_s)\,\diff s\right)\left\{\partial_t w + (\nabla_{\bx}w)\cdot\bm\alpha + \tfrac12 \Tr\left(\bm\sigma\,\bm\sigma^{\intercal}\,\nabla^2_{\bx}w\right) + \int_{\R^l} \bm\Delta_{\bm z}w(t,\bX_{t^-})\,{\bnu}(\diff t,\diff\bm z) \right\} 
    \\
    + w\,G'\left(\textstyle\int_0^t \ell(\bX_s)\,\diff s\right)\,\ell(\bX_t) = 0\,.
\end{multline}

Next, we introduce the process $V_t := 1/\tilde{w}_t$, which can be written as $V_t = v(t,\bX_{t})\,/G_t$ where $v(t,\bx) := 1/w(t,\bx)$, and
\begin{equation*}
    G_t := G\left(\textstyle\int_0^t \ell(\bX_s)\diff s\right)\,.
\end{equation*}
Thus, by the multidimensional L\'evy-It\^o formula, we have that (we omit the arguments of the functions when there is no confusion)
\begin{equation*}
    \begin{split}   
        \diff V_t &= -\frac{1}{G_t\,w^2}\Big(\partial_t w + (\nabla_{\bx}w)\cdot\bm\alpha + \tfrac12 \Tr\left(\bm\sigma\,\bm\sigma^{\intercal}\,\nabla^2_{\bx}w\right)\Big)\diff t  - \frac{1}{G_t\,w^2}(\nabla_{\bx}w)\,\bm\sigma\,\diff \bW_t 
        \\
        &\quad + \frac{1}{G_t\,w^3}\,\left| (\nabla_{\bx}w)\,\bm\sigma \right|^2 \diff t  + \frac{1}{G_t}\sum_{j\in \D} \int_{\R}\left( \frac{1}{w(t,\bX_{t^-}+\bm\gamma^{(j)})} - \frac{1}{w(t,\bX_{t^-})}\right)\,\mu^j(\diff t,\diff z)\\
        &\quad - G'\left(-\textstyle\int_0^t \ell(\bX_s)\,\diff s\right) \,\ell(\bX_t)\,\frac{1}{w}\,\diff t\,,
    \end{split}
\end{equation*}
and using the identity in \eqref{eq: identity thm v2} we have 
\begin{equation}
\label{eqn:dV-after-mtg-eqn}
    \begin{split}
        \diff V_t 
        &=
        \frac{1}{G_t\,w^2}\, \int_{\R^l} \bm\Delta_{\bm z}w(t,\bX_t)\,{\bnu}(\diff t,\diff\bm z)  
        +
        \frac{1}{w\,G^2_t}\,G'\left(\textstyle\int_0^t \ell(\bX_s)\,\diff s\right)\,\ell(\bX_t)\,\diff t 
        \\
        & \quad -
        \frac{1}{G_t\,w^2}(\nabla_{\bx}w)\,\bm\sigma\,\diff \bW_t 
        +
        \frac{1}{G_t\,w^3}\,\left| (\nabla_{\bx}w)\,\bm\sigma \right|^2  \diff t  
        \\
        & \quad +
        \frac{1}{G_t}\sum_{j\in\D} \int_{\R}\left( \frac{1}{w(t,\bX_{t^-}+\bm\gamma^{(j)})} - \frac{1}{w(t,\bX_{t^-})}\right)\,\mu^j(\diff t,\diff z)
        \\
        & \quad - 
        G'\left(-\textstyle\int_0^t \ell(\bX_s)\,\diff s\right) \,\ell(\bX_t)\,\frac{1}{w}\,\diff t\,.
    \end{split}
\end{equation}

Next, define the process $\phi_t:= Z_t\,V_t$ for $t\in[0,T]$. We claim that $\phi_t=c$ for all $t\in[0,T]$ for some constant $c\in\R$. To see this, note that 
\begin{equation*}
    \diff \phi_t = Z_{t^-}\,\diff V_t + V_{t^-}\,\diff Z_t + \diff [Z,V]_t\,,
\end{equation*}
which after direct substitution, using \eqref{eqn:dV-after-mtg-eqn}, \eqref{eqn:dZ-stoch-exp}, and the formula for $\diff [Z,V]_t$, we have
\begin{equation*}
    \begin{split}
        \diff \phi_t 
        &=
        \frac{Z_{t^-}}{G_t\,w^2}\, \int_{\R^l} \bm\Delta_{\bm z}w(t,\bX_{t^-})\,{\bnu}(\diff t,\diff\bm z)  
        + \frac{Z_{t^-}}{w\,G^2_t}\,G'\left(\textstyle\int_0^t \ell(\bX_s)\,\diff s\right)\,\ell(\bX_t)\,\diff t 
        \\
        & \quad
        \underbrace{-\; \frac{Z_{t^-}}{G_t\,w^2}(\nabla_{\bx}w)\,\bm\sigma\,\diff \bW_t  }_{(a)}\; + \;
        \underbrace{ \frac{Z_{t^-}}{G_t\,w^3}\,\left| (\nabla_{\bx}w)\,\bm\sigma \right|^2\diff t}_{(b)} 
        \\
        &\quad 
         + \frac{Z_{t^-}}{G_t}\sum_{j\in \D} \int_{\R}\left( \frac{1}{w(t,\bX_{t^-}+\bm\gamma^{(j)})} - \frac{1}{w(t,\bX_{t^-})}\right)\,\mu^j(\diff t,\diff z)
        \\
        &\quad 
        - Z_{t^-}\,G'\left(-\textstyle\int_0^t \ell(\bX_s)\,\diff s\right) \,\ell(\bX_t)\,\frac{1}{w}\,\diff t\; +\;
        \underbrace{ V_{t^-}\,Z_{t^-}\,(\nabla_{\bx}\log w)\,\bm\sigma\,\diff \bW_t }_{(c)}
        \\
        &\quad 
        + \frac{V_{t^-}\,Z_{t^-}}{w}\,\int_{\R^l} \bm\Delta_{\bm z}w(t,\bX_{t^-})\,{\tilde\bmu}(\diff t,\diff\bm z)\;\; \underbrace{ -\; \frac{Z_{t^-}}{G_t\,w^2}\Tr\Big(\bm\sigma\,\bm\sigma^{\intercal}\,(\nabla_{\bx}\log w)^{\intercal}\,\nabla_{\bx}w\Big)\diff t }_{(d)} 
        \\
        &\quad 
        - \frac{Z_{t^-} }{G_t}\sum_{j \in \D}  \int_{\R}\left(\frac{1}{w(t,\bX_{t^-}+\bm\gamma^{(j)})} 
        -
        \frac{1}{w(t,\bX_{t^-})} \right)\left(1- \frac{w(t,\bX_{t^-}+\bm\gamma^{(j)})}{w(t,\bX_{t^-})}\right)\mu^j(\diff t,\diff z)\,. 
    \end{split}
\end{equation*}
As $\nabla_{\bx} \log w = \frac{\nabla_{\bx} w}{w}$ and  $V_t = 1/(w\,G_t)$, after a short calculation we find that $(a)$ cancels with $(c)$ . Similarly, $(b)$ cancels with $(d)$ by factoring $1/w$ out of the $\Tr(\cdot)$ operator. Then, it follows that $\diff \phi_t$ reduces to
\allowdisplaybreaks
\begin{equation}
\label{eqn:dphi-2}
\begin{split}
\allowdisplaybreaks
    \diff \phi_t 
    &= 
    \frac{Z_{t^-}}{G_t\,w^2}\, \int_{\R^l} \bm\Delta_{\bm z}w(t,\bX_{t^-})\,{\bnu}(\diff t,\diff\bm z) +\, \underbrace{ \frac{Z_{t^-}}{w\,G^2_t}\,G'\left(\textstyle\int_0^t \ell(\bX_s)\,\diff s\right)\,\ell(\bX_t)\,\diff t }_{(e)} 
    \\
    &\quad 
    + \frac{Z_{t^-}}{G_t}\sum_{j \in \D} \int_{\R}\left( \frac{1}{w(t,\bX_{t^-}+\bm\gamma^{(j)})} - \frac{1}{w(t,\bX_{t^-})}\right)\,\mu^j(\diff t,\diff z)
    \\
    &\quad 
    \underbrace{-\, Z_{t^-}\,G'\left(-\textstyle\int_0^t \ell(\bX_s)\,\diff s\right) \,\ell(\bX_t)\,\frac{1}{w}\,\diff t}_{(f)}
    \;+\; \frac{V_{t^-}\,Z_{t^-}}{w}\,\int_{\R^l} \bm\Delta_{\bm z}w(t,\bX_{t^-})\,{\tilde\bmu}(\diff t,\diff\bm z)
    \\
    &\quad 
    - \frac{Z_{t^-} }{G_t}\sum_{j \in \D} \int_{\R}\left(\frac{1}{w(t,\bX_{t^-}+\bm\gamma^{(j)})} - \frac{1}{w(t,\bX_{t^-})} \right)\left(1- \frac{w(t,\bX_{t^-}+\bm\gamma^{(j)})}{w(t,\bX_{t^-})}\right)\mu^j(\diff t,\diff z)\,.
\end{split}
\end{equation}
Note that $G(-x) = 1/G(x)$ (as $G(x-x)=G(x)G(-x)$ and $G(0)=1$), hence $G'(-x) = \frac{G'(x)}{G^2(x)}$, 
and thus
\begin{equation*}
    G'\left(-\textstyle\int_0^t \ell(\bX_s)\,\diff s \right) = \frac{G'\left(\int_0^t \ell(\bX_s)\,\diff s\right)}{G^2\left(\int_0^t \ell(\bX_s)\,\diff s\right) }\,.
\end{equation*}
Using this relationship,  the $(e)$ and $(f)$ terms in \eqref{eqn:dphi-2}  cancel, in which case we have
\begin{equation*}
\begin{split}
    \diff \phi_t 
    &=
    \underbrace{\frac{Z_{t^-}}{G_t\,w^2}\,\int_{\R^l} \bm\Delta_{\bm z}w(t,\bX_{t^-})\,{\bnu}(\diff t,\diff\bm z)  }_{(g)}
    \\
    & \quad 
    + \frac{Z_{t^-}}{G_t}\sum_{j \in \D} \int_{\R}\left( \frac{1}{w(t,\bX_{t^-}+\bm\gamma^{(j)})} - \frac{1}{w(t,\bX_{t^-})}\right)\,
    \mu^j(\diff t,\diff z)
    \\
    &\quad 
    + \frac{V_{t^-}\,Z_{t^-}}{w}\,\int_{\R^l} \bm\Delta_{\bm z}w(t,\bX_{t^-})\,{\bmu}(\diff t,\diff\bm z)\;  \underbrace{ - \; \frac{V_{t^-}\,Z_{t^-}}{w}\,\int_{\R^l} \bm\Delta_{\bm z}w(t,\bX_{t^-})\,{\bnu}(\diff t,\diff\bm z)}_{(h)}
    \\
    &\quad 
    - \frac{Z_{t^-} }{G_t}\sum_{j\in \D}  \int_{\R}\left(\frac{1}{w(t,\bX_{t^-}+\bm\gamma^{(j)})} - \frac{1}{w(t,\bX_{t^-})} \right)
    \left(1- \frac{w(t,\bX_{t^-}+\bm\gamma^{(j)})}{w(t,\bX_{t^-})}\right)\mu^j(\diff t,\diff z)\,.
\end{split}
\end{equation*}
From the definition of $V_t$, we see that $(g)$ and $(h)$ cancel. Finally we have
\begin{equation*}
    \begin{split}
        \diff \phi_t 
        &=
        \frac{Z_{t^-}}{G_t}\sum_{j \in \D} \int_{\R}\left( \frac{1}{w(t,\bX_{t^-}+\bm\gamma^{(j)})} - \frac{1}{w(t,\bX_{t^-})}\right)\,\mu^j(\diff t,\diff z)
        \\
        &\quad 
        - \frac{Z_{t^-}}{G_t\,w}\,\sum_{j \in \D }\int_{\R} \left(1 - \frac{w(t,\bX_{t^-}+\bm\gamma^{(j)})}{w(t,\bX_{t^-})}\right)\,{\mu}^j(\diff t,\diff z)
        \\
        &\quad 
        - \frac{Z_{t^-} }{G_t}\sum_{j\in \D}  \int_{\R}\left(\frac{1}{w(t,\bX_{t^-}+\bm\gamma^{(j)})} - \frac{1}{w(t,\bX_{t^-})} \right)\left(1- \frac{w(t,X_{t^-}+\bm\gamma^{(j)})}{w(t,\bX_{t^-})}\right)\mu^j(\diff t,\diff z)
        \\
        &= 
        0\,.
    \end{split}
\end{equation*}
The last equality follows by collecting like terms in the preceding lines. Thus, $\phi_t = c$ for all $t\in[0,T]$, and some $c\in\R$. As $Z_t = c\,w(t,\bX_{t})\,G_t$ and  $\E[Z_T]=1$, it follows that 
\begin{equation*}
    c = \E[w(T,\bX_T)\,G_T ]^{-1}
    =
    \E\left[w(T,\bX_T)\,G\left(\textstyle\int_0^T \ell(\bX_s)\,\diff s\right) \right]^{-1}\,,
\end{equation*}
from which we obtain the required results
\begin{equation*}
    Z_t = \frac{w(t,\bX_{t})\,G_t}{\E[w(T,\bX_T)\,G_T]}\,, \quad t\in [0,T]\,,
    \quad \text{and} \quad
    Z_T = \frac{H(\bX_T)\,G_T}{\E[H(\bX_T)\,G_T]}\,.
\end{equation*}
\end{proof}

\begin{corollary}\label{cor:dQ-dP}
Let Assumptions \ref{assumption: lip and linear growth} be fulfilled, $D_{-\mfX} \neq \emptyset$, and $(\eone, \etwo) \in D_{-\mfX}$. Further let $\blambda^\dagger$, $\bh^\dagger$, and $\omega^\dagger(t,\bx)$ be as in Theorem \ref{thm: verification} and satisfying its assumptions. Then, the probability measure $\Q^\dagger$ has RN-density 
\begin{equation}\label{eq:dq-dp-rep}
    \frac{\diff \Q^\dagger}{\diff \P} 
    =
    Z^{\blambda^\dagger,\bh^\dagger}_T
    =
    \frac{\exp\left( -\eone \cdot \bm f(\bX_T) -\etwo\cdot\int_0^T  \bg(\bX_u)\,\diff u\right)
    }
    {
    \E\left[\exp\left( -\eone \cdot \bm f(\bX_T) -\etwo\cdot\int_0^T  \bg(\bX_u)\,\diff u\right)\right]}\;.
\end{equation}
\end{corollary}
\begin{proof}
By the definition of $\omega^\dagger(t, x)$ in \eqref{eq:omega-star} we have
\begin{align*}
    \omega^\dagger(t,\bx) 
    &=
    \E_{t,\bx}\left[ H(\bX_T) \;G\left( \int_t^T  \ell(\bX_u)\diff u \right) \right]\,.
\end{align*}
where we set $H(\bx) := \exp\left( -\eone \cdot \left(\bm f(\bx)-\bm c\right)\right)$, $\ell(\bx) = \etwo \cdot \bg(\bx)$, and $G(x) := \exp\left( - x + \etwo\cdot\bd  \right)$. Applying Theorem \ref{thm: form of optimal RN derivative}, the RN-density is becomes
\begin{align*}
    Z^{\blambda^\dagger,\bh^\dagger}_T 
    &= 
    \frac{\exp\left( -\eone \cdot \left(\bm f(\bX_T)-\bm c\right) -\etwo\cdot\left(\int_0^T  \bg(\bX_u) \,\diff u-\bd \right)\right)
    }{
    \E\left[\exp\left( -\eone \cdot \left(\bm f(\bX_T)-\bm c\right) -\etwo\cdot\left(\int_0^T  \bg(\bX_u)\,\diff u-\bd  \right)\right)\right]}    
    \,,
\end{align*}
which after simplification concludes the proof.
\end{proof}
Note that the above corollary states that the RN-density is a function only of the terminal value of the processes $\bX_T$ and the running costs $\int_0^T \bg(\bX_s) \diff s$. Thus, even though the RN-density $Z_T^{\blambda^\dagger,\bh^\dagger}$ was characterised by the stochastic process $\blambda^\dagger$ and the random vector field $\bh^\dagger$, it has a representation where it does not (explicitly) depend on them. As we show in the next subsection, an optimal RN-density which attains the infimum in the optimisation problem \eqref{opt-Z-lambda-h} will be of the form \eqref{eq:dq-dp-rep} for some $\blambda$ and $\bh$ and moreover it will indirectly depend on them through the constraints.

\subsection{Solution to Optimisation Problems \eqref{opt-Z-lambda-h} and \eqref{opt}}\label{sec:solution-opt}
In this section, we present the solution to the control problem \eqref{opt-Z-lambda-h} and show that, if the solution exists, it is also the unique the constrained optimisation problem \eqref{opt}. The next result states the solution to the optimisation problem \eqref{opt-Z-lambda-h}.

\begin{theorem}[Solution to \eqref{opt-Z-lambda-h}]
\label{thm: E! of lagrange}
Let Assumptions \ref{assumption: lip and linear growth} and \ref{assumption: existence! Lagrange} be fulfilled and suppose that $\blambda^*$, $\bh^*$ are as in Theorem \ref{thm: verification}, with  Lagrange multipliers $( \estarone, \estartwo)$ solving Equation \eqref{eq: condition lagrange}, and satisfy its assumptions.
Then, there exits a solution $\Q_{\blambda^*,\bh^*}$ to \eqref{opt-Z-lambda-h} which is given in Corollary \ref{cor:dQ-dP} with optimal Lagrange multipliers $( \estarone, \estartwo)$ and where $\blambda^*$, $\bh^*$ generate the measure change.
\end{theorem}
\begin{proof}
For fixed $(\eone, \etwo)\in D_{-\mfX}$, we take $\omega^\dagger(t, \bx)$, $\blambda^\dagger$, and $\bh^\dagger$ as given in \eqref{eq:omega-star} and \eqref{eqn:lambda-h-dagger}. Denote the corresponding measure by $\Q^\dagger:=\Q_{\blambda^\dagger,\bh^\dagger}$. Recall that $\frac{\diff\Q^\dagger}{\diff\P}=Z_T^{\blambda^\dagger,\bh^\dagger}=:Z_T^\dagger$. Then we may rewrite the constraints as
\begin{align}\label{eq:proof-constraints}
\begin{split}
    \E\left[Z^{\dagger}_T\,(f_j(\bX_T)-c_j)\right] & =0 \,, \quad\ \text{ for }\, j\in \mR_1 \,,
    \\
    \quad \text{and } \quad 
    \E\left[Z^{\dagger}_T\left(\textstyle\int_0^T g_i(\bX_s)\,\diff s-d_i\right)\right]&=0\,, \quad
    \ \text{ for }\,i\in \mR_2 \,.
\end{split}
\end{align}
By Corollary \ref{cor:dQ-dP} we further have that
\begin{equation*}
    Z^{\dagger}_T 
    =
    \frac{e^{ - \eone\,\cdot\,(\bm f(\bX_T)-\bc) -\etwo \cdot\,\left(\int_0^T  \bg(\bX_u) \,\diff u-\bd \right)}}{\E\left[e^{ -\eone\,\cdot\,(\bm f(\bX_T)-\bc) -\etwo\cdot\left(\int_0^T \bg(\bX_u)\,\diff u-\bd\right) }\right] }\,,
\end{equation*}
which allows to rewrite the set of equations \eqref{eq:proof-constraints} as
\begin{equation*}
    -\partial_{\eta_k} \log \E\left[e^{ - \eone\,\cdot\,(\bm f(\bX_T)-\bc) -\etwo\cdot\left(\int_0^T \bg(\bX_u)\,\diff u-\bd\right)}\right] = 0\,, 
    \quad \forall\;k \in \mR_1 \cup \mR_2\,.
\end{equation*}
The above set of equations can be compactly written as the system of equations
\begin{equation*}
    \nabla_{\ba} K_{\mfX}\left(-\ba\right) = \bm 0\,,
\end{equation*}
which, by Assumption \ref{assumption: existence! Lagrange}, has a solution, denoted here by $( \estarone, \estartwo)$. Further, if for this choice of Lagrange multipliers, the assumptions in Theorem \ref{thm: verification}  are satisfied then, by Theorem \ref{thm: verification}, the corresponding optimal controls $\blambda^*,\bh^*$ are attainable and generate the required measure change $\Q_{\blambda^*,\bh^*}$.
\end{proof}

\begin{proposition}[SDE under  $\Q^*$]
Let the conditions of Theorem \ref{thm: E! of lagrange} be fulfilled and $\Q^*= \Q_{\blambda^*, \bh^*}$ given in Theorem \ref{thm: E! of lagrange}. Then, $\bX$ satisfies the following SDE in terms of $\Q^*$-martingales
\begin{multline*}
    \diff \bX_t 
    = 
    \left(\bm\alpha(t,\bX_t) - \bm\sigma(t,\bX_t)\,\blambda_t^{*\,\intercal}   - \int_{\R^l} \bm\gamma(t,\bX_{t^-},\bm z) \left[\bh^{*\,\intercal}_t(\bm z) \odot \bnu(\diff \bm z)\right] \right)\,\diff t
    \\
     +\bm\sigma(t,\bX_t)\,\diff \bW^{\blambda^*}_t + \int_{\R^l}\bm \gamma(t,\bX_{t^-},\bm z)\,\btmu^{\bh^*}(\diff t, \diff \bm z)\,,
\end{multline*}
where $\btmu^{\bh^*} := \bmu - \bnu^{\bh^*}$,  $\bW^{\blambda^*}$ is a $\Q^*$-Brownian Motion and $ \bnu^{\bh^*}(\diff t, \diff \bm z)=(\bm 1-\bh^*_t( \bm z))^{\intercal}\odot\,\bnu(\diff \bm z)\,\diff t$ the  $\Q^*$-compensator of $\bmu$.
\end{proposition}
\begin{proof}
This follows immediately from Girsanov's Theorem and by writing \eqref{eq: X dynamics ndim} in terms of $\bW^{\blambda^*}$ and $\btmu^{\bh^*}$.
\end{proof}

\begin{theorem}[Solution to \eqref{opt}]
\label{thm:P-and-Pprime}
If the optimisation problem \eqref{opt-Z-lambda-h} has a solution, then it is unique, and moreover it is the unique solution to optimisation problem \eqref{opt}. 
\end{theorem}
\begin{proof}
By Theorem \ref{thm: E! of lagrange} a solution to optimisation problem \eqref{opt-Z-lambda-h} is $\Q^*=\Q_{\blambda^*,\bh^*}$ where $\blambda^*$, $\bh^*$, and $\omega^*(t, \bx)$ are as in Theorem \ref{thm: verification} with Lagrange multipliers $(\estarone, \estartwo)$ solving Equation \eqref{eq: condition lagrange}. By Corollary \ref{cor:dQ-dP} (multiplying and dividing by the constants), we have that
\begin{equation*}
    \frac{\diff\Q^*}{\diff\P}
    =
    \frac{\exp\left( -\estarone \cdot \left(\bm f(\bX_T) - \bc\right) -\estartwo\cdot\left(\,\int_0^T  \bg(\bX_u)\,\diff u - \bd \right)\right)}
    {\E\left[\exp\left( -\estarone  \cdot \left(\bm f(\bX_T) - \bc\right) -\estartwo \cdot\left(\,\int_0^T  \bg(\bX_u)\,\diff u - \bd\right)\right)\right]}\,.
\end{equation*}
Next, let $\tilde{\Q}$ be any probability measure that is absolutely continuous with respect to $\P$ and under which the constraints are fulfilled. Then, observe that
\begin{align*}
    \E&\left[
    \left(\frac{\diff\Q^*}{\diff\P} - \frac{\diff\tilde{\Q}}{\diff\P}\right) \log \frac{\diff\Q^*}{\diff\P}\right]
   \\
   \begin{split}
       & \quad =
        \E\left[ \left(\frac{\diff\Q^*}{\diff\P} - \frac{\diff\tilde{\Q}}{\diff\P}\right) 
         \left\{ -\estarone \cdot \left(\bm f(\bX_T)  - \bc\right)-\estartwo\cdot\left(\int_0^T  \bg(\bX_u)\,\diff u-\bd\right)\right.\right.
         \\
         &  \quad\quad
        -
        \left.\left.\log \E\left[\exp\left( - \estarone \cdot \left(\bm f(\bX_T)-\bc\right) -\estartwo\cdot\left(\int_0^T  \bg(\bX_u)\,\diff u-\bd\right)\right)\right]
        \right\}\right]       
   \end{split}
    \\
    &\quad=
    0\,.
\end{align*}
Using the above equality, the KL-divergence from $\tilde\Q$ to $\P$ can be bounded below as follows
\begin{align*}
     D_{KL}\left(\tilde{\Q} ~||~ \P \right)
     &=
     \E\left[\frac{\diff\tilde{\Q}}{\diff\P}\,\log \frac{\diff\tilde{\Q}}{\diff\P} \right]
     +
     \E\left[ \left(\frac{\diff\Q^*}{\diff\P} - \frac{\diff\tilde{\Q}}{\diff\P}\right) \log \frac{\diff\Q^*}{\diff\P}\right]
     \\
     &=
     \E\left[\frac{\diff\tilde{\Q}}{\diff\P}\,
    \left(\log \frac{\diff\tilde{\Q}}{\diff\P}
    + \log \frac{\diff\P}{\diff\Q^*} 
    \right)\right]
    +  D_{KL}\left(\Q^* ~||~ \P \right)
    \\
    &=
    \E^{\tilde{\Q}}\left[
    \log \frac{\diff\tilde{\Q}}{\diff\Q^*}
    \right]
    +  D_{KL}\left(\Q^* ~||~ \P \right)
    \\
    &= 
    D_{KL}\left(\tilde{\Q} ~||~ \Q^* \right)+  D_{KL}\left(\Q^* ~||~ \P \right)
    \\
    & \ge 
    D_{KL}\left(\Q^* ~||~ \P \right)\,.
\end{align*}
Thus, $\Q^*$ is indeed a solution to \eqref{opt}. Uniqueness of the solution to \eqref{opt} follows by strict convexity of the KL-divergence, which implies uniqueness of \eqref{opt-Z-lambda-h}.
\end{proof}

\section{Analytically Tractable Examples}\label{sec:examples}
In this section we provide examples illustrating how the dynamics of processes change when moving from $\P$ to $\Q^*$. First, we discuss the sign of the optimal Lagrange multiplier under one single constraint. Second, we consider how the solution to the optimisation problem \eqref{opt} is connected to pinned measures. Third, we provide explicit expressions for the Lagrange multipliers and the optimal RN-density under two VaR constraints. Forth, we consider a constraint on the mean when the underlying process has independent increments. Finally, we study how a Brownian motion is perturbed when we keep its mean equal to 0 but alter its standard deviation.

\subsection{Sign of Lagrange multiplier for Single Constraint}\label{sec:ex-sign-Lagrange}
For the case when there is only one constraint, i.e., $r_1 = 1$, $r_2 = 0$, or $r_1 = 0$, $r_2 = 1$, we can specify the sign of the Lagrange multiplier. For simplicity we assume that $r_2 = 0$ however the following proposition also holds for one running cost constraint.  
\begin{proposition}[Sign of Lagrange multiplier]
Let Assumptions \ref{assumption: lip and linear growth} and \ref{assumption: existence! Lagrange} be fulfilled and denote by $\eta^*$ the unique solution to \eqref{eq: condition lagrange}. Consider Problem \eqref{opt} with one constraint, i.e. $r_1 = 1$, $r_2 = 0$, which we write as $\E^\Q[f(\bX_T)] = c$. Then, the sign of the optimal Lagrange multiplier $\eta^*$ is given by ${\rm{sgn}}\left(\eta^*\right)={\rm{sgn}}\left(\E[f(\bX_T)]-c\right)$.
\end{proposition}

\begin{proof}
Using the  optimal RN-density given in Corollary \ref{cor:dQ-dP}, the optimal Lagrange multiplier $\eta^*$ fulfils 
\begin{align*}
    0 
    &= 
    \E\left[\frac{e^{- \eta^* (f(\bX_T) - c)}}{\E[e^{- \eta^* (f(\bX_T) - c)}]}\left(f(\bX_T) - c\right)\right] 
    \\[0.5em]
    &= 
    \left.\frac{\diff}{\diff\eta}  \log\left(\E\left[e^{ \eta (f(\bX_T) - c)}\right] \right) \right|_{\eta =- \eta^*}
    \\[0.5em]
    &= 
    \left.\frac{\diff}{\diff\eta} K_{f(\bX_T) - c}\left( \eta\right)\right|_{\eta =- \eta^*}\,.
\end{align*}
As the derivative of a cgf of a random variable $Y$, $\frac{d}{da}K_Y(a)$, is strictly increasing in its argument $a$, we have that $\frac{\diff}{\diff\eta} K_{f(\bX_T) - c}( \eta)|_{\eta =- \eta^*}$ is strictly decreasing in $\eta^*$. Moreover,
\begin{equation*}
    \frac{\diff}{\diff\eta} K_{f(\bX_T) - c}\left( \eta\right)\big|_{\eta = 0}
    = 
    \E[f(\bX_T)] - c\,.
\end{equation*}
Clearly, if the rhs vanishes, then $\eta^*=0$.
Further,  if $\E[f(\bX_T)]-c>0$, we must have $\eta^*>0$. Similarly, if  $\E[f(\bX_T)]-c<0$, we must have $\eta^*<0$. 
\end{proof}
The above proposition states that a constraint $ \E^{\Q^*}[f(\bX_T)] = c > \E[f(\bX_T)]$, i.e., an increase in the expected value of $f(\bX_T)$ from $\P$ to $\Q^*$, corresponds to a negative optimal Lagrange multiplier. Similarly, if the expected value of $f(\bX_T)$ is decreased from $\P$ to $\Q^*$ $( \E^{\Q^*}[f(\bX_T)] = c <  \E[f(\bX_T)])$, then $\eta^*$ is positive.

\subsection{Pinned Measures}\label{sec:ex-pinned-measures}

Consider a Borel measurable set $B\in\mathcal B(\R^n)$ and the constraint function $f(\bx) = \Id_{\{\bx\in B\}}$ which results in the constraint $\Q(\bX_T\in B)=q_k$, where we choose $q_k:=1-\frac1k$, $k\in\Z_+\setminus \{0\}$. For each $k $, the optimal probability measures is
\begin{equation}
\label{eqn:alomst-pinned-dqdp}
    \frac{\diff\Q^*_k}{\diff\P} = \frac{e^{-\eta^*_k}\,\Id_{\{\bX_T\in B\}} + \Id_{\{\bX_T\notin B\}}}{e^{-\eta^*_k}\,p_B + (1-p_B)}\,,
\end{equation}
where $p_B:=\P(\bX_T\in B)$ and $\eta^*_k$ is such that the constraint is binding. We include the subscript index on $\Q^*_k$ as we aim to consider the limiting measure for $k \uparrow  \infty$. By enforcing the constraint, we have that
\begin{equation*}
    \E\left[ \frac{e^{-\eta^*_k}\,\Id_{\{\bX_T\in B\}} + \Id_{\{\bX_T\notin B\}}}{e^{-\eta^*_k}\,p_B + (1-p_B)}\;\Id_{\{\bX_T\in B\}}\right] = q_n,
\end{equation*}
which gives
\begin{equation}\label{eq:pinned-eta}
    \eta^*_k = -\log \left( \frac{p_B}{q_k} \frac{1-q_k}{1-p_B}\right)\,.
\end{equation}
Substituting \eqref{eq:pinned-eta} into Equation \eqref{eqn:alomst-pinned-dqdp}, the RN-density becomes
\begin{equation*}
    \frac{\diff\Q^*_k}{\diff\P} = 
    \left(\tfrac{1-\frac1k}{p_B}\right)\Id_{\{\bX_T\in B\}} 
    + \left(\tfrac{1}{k(1-p_B)}\right) \Id_{\{\bX_T\notin B\}}\,.
\end{equation*}
The limiting measure induced by the RN-density 
$\frac{\diff\Q^*}{\diff\P}
:=\lim_{k\to\infty}\frac{\diff\Q^*_k}{\diff\P}
=\frac{\Id_{\{\bX_T\in B\}}}{\P(\bX_T\in B)}$  
coincides with the so-called pinned measures. Pinned measures are those for which the terminal value of the process $\bX$ must lie within the set $B$. Note this limiting measure is not equivalent to $\P$, but absolutely continuous $\Q\ll\P$.

\subsection{VaR Constraints}\label{sec:es-var}
We consider the constraint functions $f_1(\bx) = \Id_{\{x^j \le q_1\}}$ and $f_2(\bx) = \Id_{\{x^j \le q_2\}}$ with $q_1, q_2 \in \R$ such that  $\text{essinf} \,X_T^j < q_1 < q_2 < \text{esssup}\, X_T^j$ and $j\in\{1,2,\dots,n\}$. The corresponding constraints are $\Q(X_T^j \le q_i) = \beta_i$, $0 < \beta_1 < \beta_2 <1$, $i =1,2$, which are Value-at-Risk\footnote{For a univariate random variable $Y$ and a probability measure $\Q$, the $\Q$-Value-at-Risk at level $\beta \in (0,1)$ is defined as $\text{VaR}^\Q_\beta(Y):= \inf\left\{y \in \R ~|~ \Q\left(Y \le y\right) \ge \beta\right\}$. For simplicity of notation we write $\VaR_\beta(Y):=\VaR_\beta^\P(Y)$.} (VaR) constraints at levels $\beta_i$ if the $\P$-distribution of $X_T^j$ is continuous. Next, we rewrite the second constraint as
\begin{align*}
\Q(q_1 < X_T^j \le q_2) &= \beta_2 - \beta_1
\end{align*}
and by Corollary \ref{cor:dQ-dP} the RN-density for fixed Lagrange multipliers $\eta_1, \eta_2$ becomes
\begin{equation}\label{eq:RN-2_VaR}
    \frac{\diff\Q}{\diff\P}
    = \frac{e^{-\eta_1} \Id_{\{X_T^j \le q_1\}} + e^{-\eta_2} \Id_{\{q_1 < X_T^j \le q_2\}} + 
    \Id_{\{ X_T^j > q_2\}}}
    {C(\eta_1, \eta_2)}\,,
\end{equation}
where $C(\eta_1, \eta_2):=e^{-\eta_1} \P(X_T^j \le q_1) + e^{-\eta_2} \P(q_1 < X_T^j \le q_2) + 
\P( X_T^j > q_2)$ is the normalising constant. Further, the optimal Lagrange multipliers $\eta_1^*$ and $\eta_2^*$ satisfy
\begin{align*}
    \beta_1
    &= \frac{e^{-\eta_1^*}\P(X_T^j \le q_1)}{C(\eta_1^*, \eta_2^*)}
    \quad \text{and} \quad
    \beta_2- \beta_1
    = \frac{e^{-\eta_2^*}\P(q_1 < X_T^j \le q_2)}{C(\eta_1^*, \eta_2^*)}\,.
\end{align*}
Inserting this into \eqref{eq:RN-2_VaR} the optimal RN-density is
\begin{equation*}
    \frac{\diff\Q^*}{\diff\P}
    =
    \tfrac{\beta_1}{\P(X_T^j \le q_1)}\Id_{\{X_T^j \le q_1\}}
    +
    \tfrac{\beta_2 - \beta_1}{\P(q_1 < X_T^j \le q_2)}\Id_{\{q_1 < X_T^j \le q_2\}}
    +
    \tfrac{1-\beta_2}{\P(X_T^j > q_2)}\Id_{\{X_T^j > q_2\}}\,.
\end{equation*}
Moreover, the optimal Lagrange multipliers are given by
\begin{align*}
    \eta_1^*
    &=
    \log\left(
    \tfrac{\P\left(X_T^j \le q_1\right)}{\beta_1}\; 
    \tfrac{1 - \beta_2}{\P\left(X_T^j > q_2\right)}
    \right)\,,
    \\
    \eta_2^*
    &= 
    \log\left(
    \tfrac{\P\left(q_1 < X_T^j \le q_2\right)}{\beta_2 - \beta_1}\;
    \tfrac{1 - \beta_2}{\P\left( X_T^j > q_2\right)}
    \right)\,,
\end{align*}
and the normalising constant simplifies to $C(\eta_1^*, \eta_2^*) = \frac{\P\left(X_T^j> q_2\right)}{1 - \beta_2}$.

Note that the explicit formulas for the Lagrange multipliers and the RN-density hold for any L\'evy-It\^o process  $(\bX_t)_{t \in [0,T]}$. Moreover, the assumptions on $q_1, q_2, \beta_1$, and $\beta_2$ are enough to guarantee existence of the solution to \eqref{opt} with VaR constraints.

\subsection{Linear Constraint Function for Process with Independent Increments}\label{sec:ex-ind-increm}
For simplicity we consider a one-dimensional process and a linear constraint function. That is, we let  $(X_t)_{t\in [0,T]}$ be the solution to the SDE under $\P$
\begin{equation*}
    \diff X_t = \mu(t)\, \diff t + \sigma(t)\, \diff W_t + \int_{\R}z\,\tilde \mu(\diff t, \diff z)\,,
\end{equation*}
where $W$ is a one-dimensional $\P$-Brownian motion, $\mu(\diff t, \diff z)$ the Poisson random measure describing Poisson arrivals of independent and identically distributed marks, and $X_0=x_0 \in \R$. We consider optimisation problem \eqref{opt} with a constraint on the expected value of $X_T$, i.e.  $\E^\Q[X_T]=c$, $c\in \R$. Note that this constraint encompasses linear constraint functions $f(x) = a_1\,x+a_2$, with $a_1, a_2\in \R$, $a_1 \neq 0$, as, for this choice of $f$, the constraint $\E^\Q[f(X_T)]=c$ is equivalent to $\E^\Q[X_T]=(c-a_2)/a_1$. 

From Theorem \ref{thm: verification} for fixed Lagrange multiplier $\eta \in D_{-(X_T - c)}$, and as  $X$ has independent increments, we have 
\begin{align*}
    \omega(t,x) &= \E_{t,x}\left[e^{-\eta\,(X_T-c)}\right]
    = e^{-\eta\,(x-c)}\,\E\left[e^{-\eta\,(X_T-X_{t})}\right]\,,
    \\
    \lambda(t,x) & = \eta\,\sigma(t)\,,
    \quad \text{and}\quad
    \\
    h_t(z) &= 1 - e^{-\eta\,z}\,.
\end{align*}
We note that $\omega\in\mathcal{C}^{1,2}([0,T]\times\R)$ and that $\lambda$, $h$ induce $\Q^{\lambda,h}\in\mcQ$.

If we further assume that $\nu(\diff z,\diff t) = \ell\;\Phi_{a,b}(\diff z)\,\diff t$ where $\ell>0$ is the rate parameter of the Poisson process and $\Phi_{a,b}(z):=\Phi((z-a)/b)$ is the distribution function of the marks -- here $\Phi$ is the distribution function of a standard normal random variable. Then, for fixed $\eta$ and $\Q^{\lambda,h} \in \mcQ$ the constraint equation
$\E^{\Q^{\lambda,h}}[X_T]=c$
becomes (after some calculations)
\begin{align}
\label{eq:example-mean-indep-increments}
    x_0 
    +
    A
    -
    \eta\,\Sigma^2 
    - \ell\,T\,\eta\,b^2\,e^{-a\,\eta + \tfrac12 \eta^2\,b^2} - \ell\,a\,T 
    =c\,,
\end{align}
where $A:=\int_0^T \alpha(t)\,\diff t$ and $\Sigma^2:=\int_0^T \sigma^2(t)\diff t$.
The optimal Lagrange multiplier $\eta^*$ that binds the constraint exists since the lhs of \eqref{eq:example-mean-indep-increments} is continuous in $\eta$ and diverges to $-\infty$ for $\eta\to  \infty$ and diverges to $\infty$ for $\eta\to - \infty$. Uniqueness of $\eta^*$ follows by uniqueness of the solution to \eqref{opt}.

\subsection{Brownian Motion with Arbitrary Variance}\label{sec:ex-Brownian}
Let $X_t = W_t$, $t \in [0,T]$, be a one-dimensional $\P$-Brownian motion and consider the constraints $\E^\Q[X_T] = 0$ and $\E^\Q[X^2_T]=\kappa\,T$, for $\kappa>0$, $
\kappa\ne1$. Note that since $\E[X_T] = 0$ and $\E[X^2_T]=T$ the constraints result that under the optimal probability measure, the mean of $X_T$ is kept fixed to its $\P$ value while the variance is scaled by $\kappa$. For Lagrange multipliers $\eta_1, \eta_2 \in D_{-\mfX}$ we have by Theorem \ref{thm: verification} that 
\begin{align*}
    \omega(t,x) 
    &=
    \E_{t,x}\left[e^{-\eta_1\,X_T -\eta_2\,(X^2_T-\kappa\,T)}\right]\,
    \quad \text{and}
    \\
    \lambda(t,x) 
    &=
    \frac{2\,\eta_2}{2\,\eta_2(T-t)+1}\,x + \frac{\eta_1}{2\,\eta_2(T-t)+1}\,.
\end{align*}
Therefore, $(X_t)_{t\in [0, T]}$ satisfies the following SDE under $\Q_\lambda$
\begin{equation*}
    \diff X_t =  a_t\,(\beta - X_t)\,\diff t + \diff W^{\lambda}_t\,,
\end{equation*}
where $W^{\lambda}$ is a $\Q_\lambda$-Brownian motion, $a_t :=  2\, \eta_2 / (2\,\eta_2(T-t) +1)$, and $\beta := -\eta_1/(2\,\eta_2)$ with $\eta_1,\eta_2$ to be determined to bind the constraints. By employing It\^o's formula on the process 
\begin{equation*}
Y_t:=X_t \,e^{\int_0^t a_s\,\diff s}\,, \quad t\geq 0\,,
\end{equation*}
we obtain that 
\begin{equation*}
X_t  = -\eta_1\int_0^t \frac{\,2\,\eta_2\,(T-t)+1}{(2\,\eta_2\,(T-s)+1)^2}\,\diff s  + \int_0^t \frac{2\,\eta_2\,(T-t)+1}{2\,\eta_2\,(T-s)+1}\,\diff W^{\lambda}_s\,.
\end{equation*}
As the coefficient of the It\^o integral  is deterministic, $X_T$ is normally distributed, and as the It\^o integral has zero $\Q^\lambda$-mean, we see that $\eta^*_1=0$ to enforce the mean constraint. Using It\^o's isometry we find that $\eta^*_2 = (1-\kappa)/(2\,\kappa\,T)$ is required to enforce the variance constraint. Note that if $\kappa\in(0,1)$ -- a reduction of the variance under $\Q^*$--, then $\eta^*_2>0$, which implies that the process $(X_t)_{t\in [0, T]}$ mean-reverts around zero to reduce the variance; similarly, if $\kappa>1$ -- an increase of the variance under $\Q^*$ --, then $\eta^*_2<0$, which implies that the process $(X_t)_{t\in [0, T]}$ is mean-avoiding to increase the variance. Finally, under the optimal measure $\Q^*$ the process
$(X_t)_{t\in [0, T]}$ satisfies the SDE 
\begin{equation*}
    \diff X_t = - \left(\tfrac{T}{1-\kappa}-t\right)^{-1}\,  X_t\,\diff t + \diff W^*_t\,,
\end{equation*}
where $W^*$ is a $\Q^*$-Brownian motion.
The above SDE shows that $X$ is a Ornstein-Uhlenbeck process. Note the coefficient of the drift remains finite for all $t\in[0,T]$.

\section{Infinitesimal Perturbations}
In this section we consider small perturbations, that is constraints where the $\Q$-expectations equal their $\P$-expectations plus $\ep\,\bdelta$, where $\bdelta$ is the direction of the perturbation and $\ep $ is small. We prove, in this setup, that the optimal Lagrange multiplier is, up to order $o(\ep)$, the inverse of the $\P$-covariance matrix of the constraint functions multiplied by $\ep$ and the direction of the perturbation. Using the results on infinitesimal perturbations, we define a directional derivative -- termed \textit{entropic derivative} -- of risk functionals in the direction of least relative entropy in Section \ref{sec:entropic-derivative}.

\subsection{Optimal Lagrange Multiplier}\label{sec:small-perturbation-lagrange}
For functions $f_j,g_i:\R^n\to\R$ with $j \in \mR_1$ and $i \in \mR_2$, we define the random vector
\begin{align}\label{eq:bmff}
\begin{split}
    \bmff = (\mff_1, \ldots, \mff_{r_1 + r_2})\,, 
    \qquad \text{where} \qquad
    \mff_j&:=f_j(\bX_T)\,, \;\forall\, j \in\mR_1\,, 
    \quad \text{and} \quad
    \\
    \mff_i&:=\int_0^T g_i(\bX_s)\,\diff s\,,\; \forall \,i \in \mR_2\,.
\end{split}
\end{align}
We assume throughout this section that $\C(\mff_i,  \mff_j) < \infty$, where $\C$ denotes the $\P$-covariance, for all $i,j \in \mR_1\cup \mR_2$. Using the above notation, we state the optimisation problem concerning infinitesimal perturbations.
\begin{optimisation}
For $\ep\neq 0$, $\bdelta := (\delta_1, \ldots, \delta_{r_1 + r_2}) \in \R^{r_1 + r_2}$, and a random vector $\bmff$ given in \eqref{eq:bmff}, we consider the optimisation problem
\begin{equation}\label{opt-epsilon}
    \inf_{\Q \ll \P}\; D_{KL}(\Q ~||~\P)
    \quad \text{subject to}
    \quad
    \E^{\Q}\left[\bmff\right]=\E^{\P}\left[\bmff\right] + \ep \, \bdelta\,,
    \tag{$P_\ep$}
\end{equation}
where the constraints are understood as a system of equations.
\end{optimisation}
If a solution to optimisation problem \eqref{opt-epsilon} exists, we denote the probability measure attaining the infimum by $\Q_\ep^*$. This probability measure may be viewed as arising from small perturbations of $\bmff$ in direction of $\bdelta$. The next result shows that the optimal Lagrange multiplier $\etaep$ is, up to $o(\ep)$, equal to the inverse of the $\P$-covariance matrix of the constraint functions multiplied by $\ep$ and the direction of the perturbation $\bdelta$. We first prove this result for a single constraint and then, using slightly stronger assumptions, for a collection of constraints.  

\begin{theorem}[Single Constraint] \label{thm:perturbation-one-const}
Let $r_1 +r_2= 1$, and write $\mfX = \mff - \E[\mff] - \ep \delta$, i.e. we only consider one constraint, and assume that $\mff$ satisfies $\C\big(\mff, \left(\mff - \E[\mff]\right)^2\big) \in \R / \{0\}$. 
Under the Assumptions of Theorem  \ref{thm: E! of lagrange},
the optimisation problem \eqref{opt-epsilon} has a unique solution $\Q^*_\ep$ with Lagrange multiplier $\etaepone \in \R $ satisfying
\begin{equation*}
    \etaepone = - \frac{1}{\V(\mff)}\delta\;\ep + o(\ep)\,,
\end{equation*}
where $\V(\mff) := \E\left[(\mff - \E[\mff])^2\right]$ denotes the $\P$-variance of $\mff$. Moreover, the KL-divergence of $\Q^*_\ep$ with respect to $\P$ is
\begin{equation*}
    D_{KL}(\Q^*_\ep ~||~\P)
    =
    \frac{1}{\V(\mff)}\, \delta^2\,\ep^2 +o(\ep^2) \,.
\end{equation*}
\end{theorem}

\begin{proof}
By Theorem \ref{thm: E! of lagrange}, the solution to problem \eqref{opt-epsilon} is 
\begin{equation*}
    \frac{\diff\Q^*_\ep}{\diff\P} 
    = 
    \frac{
    \displaystyle\exp\left( -\etaepone \left(\mff- \E[\mff] - \ep \delta\right) \right) 
    }
    {
    \displaystyle\E\left[ \exp\left( -\etaepone   \left(\mff- \E[\mff] - \ep \delta\right) \right)  \right]
    }
    = 
    \frac{
    \displaystyle\exp\left( -\etaepone \,\mff \right) 
    }
    {
    \displaystyle\E\left[ \exp\left( -\etaepone \,   \mff \right)  \right]
    }\,,
\end{equation*}
where the Lagrange multiplier $\etaepone$ solves the equation $\E^{\Q_\ep^*}\left[\mff\right] =\E[\mff] + \ep\,\delta$.
Let $\etaepone = - \frac{1}{\V(\mff)} \delta\ep  + \repone$, where $\repone := r(\ep)$ and $ r\colon \R \to \R$ is the error term satisfying $\lim_{\ep \to 0}\repone = 0$. Then for $\etaepone$ to bind the constraint, we require that
\begin{equation}\label{eq:constraint-ep-single}
\E\left[ 
    \left(\frac{
    \displaystyle\exp\left( - \etaepone  \mff\right) 
    }
    {
    \displaystyle\E\left[ \exp\left( -  \etaepone  \mff \right)  \right]
    }\;   - 1
    \right)
    \mff
    \right] 
    =
    \E\left[\left(h\left(\eta_\ep^*\right) - 1\right)\mff\right]
    =
    \ep\,\delta\,,
\end{equation}
where we set the random variable $h(z):= e^{-z\, \mff} / \E[e^{-z \, \mff}]$. Next, we show that $\repone$ is an error term of order $o(\ep)$. For this we first apply the Taylor theorem to calculate the Taylor approximation of $h$ around $\ep = 0$, and obtain that for all $\omega \in\Omega$
\begin{equation*}
    h(\etaepone) 
    =
    1 - \left(\mff - \E[\mff]\right) \left(-\frac{1}{\V(\mff)} \delta \ep + \repone\right) + \Rep(\xi)\,,
\end{equation*}
where the random variable $\Rep$ is the error term of the Taylor approximation of $h$ and $\xi$ (potentially depending on $\omega$) lies between $0$ and $-\frac{1}{\V(\mff)} \delta \ep + \repone$. Inserting the expansion of $h$ into \eqref{eq:constraint-ep-single} we obtain 
\begin{equation*}
    \E\left[\left(\mff - \E[\mff]\right) \,\mff\right]
    \left(\frac{1}{\V(\mff)} \delta\, \ep - \repone\right) + \E\left[\mff\, \Rep(\xi)\right]
    = \ep\, \delta\,.
\end{equation*}
Noting that $\E\left[\left(\mff - \E[\mff]\right)\,\mff\right] = \V(\mff) $, the above becomes
\begin{equation}\label{eq:repone}
    \V(\mff) \,\repone
     =
     \E\left[\mff \Rep(\xi)\right]\,.
\end{equation}
Using the Lagrange form of the error term, we may write
\begin{align*}
    \E\left[\mff\,\Rep(\xi)\right]
    &=
    \tfrac12  \, \E\left[\mff \, h''(\xi)\right] \left(-\frac{1}{\V(\mff)} \delta \,\ep + \repone \right)^2
    \\
    &= 
    \tfrac12 \,M(\xi)  \left(\frac{1}{\V(\mff)^2} \delta^2 \ep^2 - 2\,\frac{1}{\V(\mff)} \delta\,\ep\,\repone + \repone^2\right) 
    \,,
\end{align*}
where we set $M(\xi) := \E\left[\mff\, h''(\xi)\right]$, where $h''$ is the second derivative of $h(z)$ with respect to $z$. Note that $M$ is continuous so that $\lim_{a \to 0}M(a) =  M(0)$. Moreover, calculations show that $M(0) = \C\big(\mff, \left(\mff - \E[\mff]\right)^2\big)$ which by assumption implies that $M(0) \neq 0$ and $|M(0)| <  \infty$. 

Therefore, Equation \eqref{eq:repone} admits the quadratic form
\begin{align*}
    r_\ep^2 - 2 \repone \left( \frac{1}{\V(\mff)}\, \delta \ep +  \frac{\V(\mff)}{M(\xi)} \right) + \frac{1}{\V(\mff)^2} \delta^2 \ep^2
    =
    0\,,
\end{align*}
which has solutions
\begin{align*}
    \repone
    &=
    \frac{1}{\V(\mff)}\, \delta\, \ep +  \frac{\V(\mff)}{M(\xi)} \pm 
    \frac{\V(\mff)}{M(\xi)}\sqrt{ 1+ 2 \,\frac{M(\xi)}{\V(\mff)^2} \,\delta\,\ep\; }\,.
\end{align*}
Note that the positive root, $r_\ep^+$, is not a viable solution since is satisfies
\begin{equation*}
    \lim_{\ep \to 0}r_\ep^+
    = 
    2 \lim_{\ep \to 0 } \frac{\V(\mff)}{M(\xi)} \neq 0\,,
\end{equation*}
which contradicts that $\rep$ must converge to 0 as $\ep \to 0$. For the negative root, $r_\ep^-$, we apply Taylor's theorem for $\ell(z) = \sqrt{1 + z}= 1 + \tfrac12z + o(z)$, and obtain, recall that $\lim_{\ep \to 0} |M(\xi)| =\lim_{a \to 0} |M(a)|   <  \infty$,
\begin{equation*}
    r_\ep^-
    =
    \frac{1}{\V(\mff)}\, \delta \ep +  \frac{\V(\mff)}{M(\xi)} -
    \frac{\V(\mff)}{M(\xi)}\left( 1+  \frac{M(\xi)}{\V(\mff)^2} \delta \ep + o(\ep M(\xi)) \right)
    = o(\ep)\,.
\end{equation*}
Thus, indeed $\etaepone =-  \frac{1}{\V(\mff)} \delta \ep + o(\ep)$. 

Next, we calculate the RN-density. For this we first use Taylor's theorems for $e^{z}=1+z +o(z)$, and then, in the third equation, Taylor's theorem for $\frac{1}{1-z}=1+z+o(z)$,
\begin{align*}
\frac{\diff\Q^*_\ep}{\diff\P} 
    &= 
    \frac{
    \exp\left( \,\ep\,\frac{1}{\V(\mff)} \delta  \,\mff + o(\ep)  \right) 
    }{
    \E\left[ \exp\left( \,\ep\,  \frac{1}{\V(\mff)} \delta  \,\mff  + o(\ep) \right)  \right]}
    = \frac{1 + \ep \frac{1}{\V(\mff)} \delta  \mff+ o(\ep)}
    {\E\left[1 + \ep \frac{1}{\V(\mff)} \delta  \mff+ o(\ep)\right]}
    \\
    &= 
    \left(1 + \ep \frac{1}{\V(\mff)} \delta  \mff+ o(\ep)\right)\, \left(1 - \ep \frac{1}{\V(\mff)} \delta  \E\left[\mff\right]+ o(\ep)\right)
    \\
    &= 
    1 + \ep\, \delta\, \frac{\left(\mff  -   \E\left[\mff\right] \right)}{\V(\mff)}   + o(\ep)\,.
\end{align*}
Finally, we calculate the KL-divergence, using Taylor's theorem for $\log(1- z) = - z + o(z)$,
\begin{align*}
    D_{KL}(\Q^*_\ep~||~\P)
    &= 
    \E\left[\left(1 + \ep  \delta  \, \frac{\left(\mff  -   \E\left[\mff\right] \right)}{\V(\mff)}+ o(\ep)\right) 
    \left( \ep \delta  \, \frac{\left(\mff  -   \E\left[\mff\right] \right)}{\V(\mff)} + o(\ep)\right)\right]
    \\
    &= 
    \frac{1}{\V(\mff)} \delta^2 \ep^2 + o(\ep^2)\,.
\end{align*}
\end{proof}

We prove the result for multiple constraints using a different proof which requires slightly stronger assumptions. 

\begin{theorem}[Multiple Constraints]\label{thm:perturbation-mult-const}
Let the Assumptions of Theorem  \ref{thm: E! of lagrange} be fulfilled. Further, assume that the Lagrange multiplier $\etaep = (\eta_{\ep, 1}^*, \ldots, \eta_{\ep, r_1+r_2}^*)\in \R^{r_1 + r_2} $ corresponding to optimisation problem \eqref{opt-epsilon} is component-wise differentiable in $\ep$ and that $\E[\mff_i\mff_j\mff_k\mff_l]<\infty$ for all $i,j,k,l\in \mR_1\cup \mR_2$. Then, optimisation problem \eqref{opt-epsilon} has a unique solution $\Q^*_\ep$ with Lagrange multiplier $\etaep$ satisfying
\begin{equation*}
    \etaep = -\bm C^{-1}\bm\delta\;\ep + o(\ep)\,,
\end{equation*}
where the matrix $\bC$ has components $\bC_{ji}:=\C[\mff_j,\mff_i]$, $i,j\in \mR_1\cup \mR_2$. Furthermore, the KL-divergence of $\Q^*_\ep$ with respect to $\P$ is
\begin{equation*}
    D_{KL}(\Q^*_\ep ~||~ \P)
    =
    \bm\delta^\intercal\, \bC^{-1}\,\bm\delta\, \ep^2 + o(\ep^2) \,.
\end{equation*}
\end{theorem}
\begin{proof}
By Theorem \ref{thm: E! of lagrange}, the solution to problem \eqref{opt-epsilon} is 
\begin{equation*}
    \frac{\diff\Q^*_\ep}{\diff\P} 
    = 
    \frac{
    \displaystyle\exp\left( -\etaep \cdot \bmff \right) 
    }
    {
    \displaystyle\E\left[ \exp\left( -\etaep \cdot \bmff \right)  \right]
    }\,,
\end{equation*}
where the Lagrange multipliers $\etaep $ solve the system of equations $\E^{\Q_\ep^*}\left[\bmff\right] =\E[\bmff] + \ep\, \bdelta$. The constraints impose that for  all $k \in \mR_1 \cup \mR_2$
\begin{equation}\label{eq:constraint-ep}
\E\left[ 
    \left(\frac{
    \displaystyle\exp\left( - \etaep \cdot \bmff\right) 
    }
    {
    \displaystyle\E\left[ \exp\left( -  \etaep \cdot \bmff \right)  \right]
    }\;   - 1
    \right)
    \mff_k
    \right] 
    =
    \E\left[\left(h\left(\bm \eta_\ep^*\right) - 1\right)\mff_k\right]
    =
    \ep\,\delta_k\,,
\end{equation}
where the random variable $h(\bz):= e^{-\bz \cdot \bmff} / \E[e^{-\bz \cdot \bmff}]$ for $\bz \in \R^{r_1 + r_2}$. Let $\etaep = - \bC^{-1} \bdelta\ep  + \rep$, where $\rep :=\bm r(\ep)$, $\bm r\colon \R \to \R^{r_1 + r_2}$, is the error term satisfying component-wise $\lim_{\ep \to 0} \rep = \bm 0$. Next, we show that $\rep$ is of order $o(\ep)$. For this we first apply Taylor's theorem  to calculate the Taylor approximation of $h$. Indeed, for all $\omega \in\Omega$
\begin{equation}\label{eq:taylor-h-multi}
    h(\etaep) 
    =
    1 - \left(\bmff - \E[\bmff]\right)^\intercal \left(-\bC^{-1} \bdelta \ep + \rep\right) + R(\etaep)\,,
\end{equation}
where the random variable $R$ is the error term of the Taylor approximation of $h$, i.e. in integral form
\begin{equation*}
     R(\bx)
    = 
    \tfrac12 \sum_{i,j=1}^{r_1+r_2} \,x_i\,x_j\, \int_0^1 (1-t)\,h_{i,j}(t\,\bx) \,\diff t,
\end{equation*}
where $h_{i,j}(\bx) := \partial_{x_i x_j} h(\bx)$.
Inserting \eqref{eq:taylor-h-multi}, the expansion of $h$, into \eqref{eq:constraint-ep} we obtain for all $k\in \mR_1 \cup \mR_2$
\begin{equation}\label{eq:constaint-to-simplify}
    \E\left[\left(\bmff - \E[\bmff]\right)^\intercal\mff_k\right]
    \left(\bC^{-1} \bdelta \ep - \rep\right) + \E\left[\mff_k R(\etaep)\right]
    = \ep \delta_k\,.
\end{equation}
Note that $\E\left[\left(\bmff - \E[\bmff]\right)^\intercal\mff_k\right] = \bC_{k \,\cdot}$, where $\bC_{k\, \cdot}$ is the $k$-th row of $\bC$, and $\bC_{k \cdot} \,\bC^{-1} \bdelta =\delta_k$. Thus, Equation \eqref{eq:constaint-to-simplify} becomes
\begin{align}
    \bC_{k\, \cdot} \,
     \rep  
     &=
     \E\left[\mff_k R(\etaep)\right].
     \label{eqn:rep-in-terms-of-R}
\end{align}
Note that
\begin{align*}
    \frac{\diff}{\diff \ep} R(\etaep) 
    = 
    \sum_{l=1}^{r_1+r_2} \partial_{x_l} R(\bx)|_{\bx = \etaep}\; \frac{\diff}{\diff\ep} \eta^*_{\ep,l}
    = \sum_{l=1}^{r_1+r_2}Q_l(\etaep) \left(-(\bC^{-1} \bdelta)_l + \frac{\diff}{\diff \ep} r_{\ep,l}\right),
\end{align*}
where 
\begin{align*}
    Q_l(\bx):=&
    \left\{
    \sum_{i=1}^{r_1+r_2} x_{i}\int_0^1 (1-t) h_{i,l}(t\,\bx)\,\diff t
    \right.
    \\
    & \qquad\quad
    \left.
    + 
    \tfrac12 \sum_{i,j=1}^{r_1+r_2} x_{i}x_{j}
    \int_0^1 (1-t) h_{i,j,l}(t\,\bx)\,\diff t
    \right\},
\end{align*}
and $h_{i,j,l}(\bx):=\partial_{x_i,x_j,x_l}h(\bx)$. 
Next, define the matrix $\bB_\ep$ whose entries are $\bB_{\ep,i,j} = \E[ \mff_i Q_j(\etaep)]$. Taking derivative of \eqref{eqn:rep-in-terms-of-R} with respect to $\ep$, then stacking the equations and isolating for $\frac{\diff}{\diff\ep}\rep$, we have
\begin{equation*}
    \tfrac{\diff}{\diff\ep}\rep = -(\bC - \bB_\ep)^{-1} \bB_\ep \bC^{-1} \bdelta.
\end{equation*}
Further, after some tedious computations, and defining $\Delta \mff_i:= \mff_i - \E[\mff_i]$,
\begin{align*}
    \lim_{\ep\to0} h_{i,j,k}(\ep\,\bx) 
    =&
    \E\left[\mff_i \left(\mff_j 
    +\E[\mff_j \Delta\mff_k]]\right)\right]
    + \E\left[\mff_i\Delta\mff_j\Delta\mff_k\right]
    + \E\left[\mff_i\Delta\mff_j\right] \Delta\mff_k
    \\
    &+(\mff_i-\E[\mff_i\Delta\mff_k]) \Delta\mff_j
     + \Delta\mff_i(\mff_j + \E[\mff_j \Delta\mff_k])
     - \Delta\mff_i\Delta\mff_j\Delta\mff_k.
\end{align*}
Hence, under the assumption of bounded fourth moments, $\lim_{\ep\to0}\bB_\ep = 0$ and applying L'H\^opital's rule, we obtain (component-wise)
\begin{align*}
    \lim_{\ep \to 0}\, \frac{\rep }{\ep}
    =
    \lim_{\ep \to 0} \;\tfrac{\diff }{\diff\ep}\;\rep = \bm 0\,,
\end{align*}
which implies that $\bm\eta^*_\ep = - \bC^{-1} \bdelta \ep\, + o(\ep)$. 

To calculate the KL-divergence, we first calculate the RN-density, using similar steps to the proof of Theorem \ref{thm:perturbation-one-const},
\begin{align*}
\frac{\diff \Q^*_\ep}{\diff\P} 
    &= 
    \frac{
      1 + \ep\, \bC^{-1} \bdelta \cdot  \bmff  + o(\ep)
    }{
    \E\left[   1 + \ep\, \bC^{-1} \bdelta \cdot  \bmff   \right] + o(\ep)
    }
    \\
    &= 
    \left( 1 + \ep\, \bC^{-1} \bdelta \cdot  \bmff  + o(\ep)\right)
    \left(1 - \ep\, \bC^{-1} \bdelta \cdot  \E\left[ \bmff \right] + o(\ep)\right)
    \\
    &= 
    1 + \ep\, \bC^{-1} \bdelta \cdot  \left(\bmff - \E\left[\bmff \right]\right) + o(\ep)\,.
\end{align*}
Thus, the KL-divergence becomes
\begin{align*}
    D_{KL}(\Q^*_\ep~||~\P)
    &= 
    \E\left[\left( 1 + \ep\, \bC^{-1} \bdelta \cdot  \left(\bmff - \E\left[\bmff \right]\right) + o(\ep)\right)\left( \ep\, \bC^{-1} \bdelta \cdot  \left(\bmff - \E\left[\bmff \right]\right) + o(\ep)\right)\right]
    \\
    &= 
    \ep^2 \bdelta^\intercal \bC^{-1}\E\left[\left(\bmff - \E\left[\bmff \right]\right)^2\right] \bC^{-1} \bdelta
    + o\left(\ep^2\right)
        \\
    &= 
    \ep^2 \bdelta^\intercal \bC^{-1} \bdelta
    + o(\ep^2)\,,
\end{align*}
which concludes the proof.
\end{proof}

\subsection{Entropic Derivative}\label{sec:entropic-derivative}
In this section, we define a derivative of a risk functional along constraints in the direction of least relative entropy. For this we use the same notation as in Section \ref{sec:small-perturbation-lagrange} and denote by $\bmff$ a vector of constraints given in \eqref{eq:bmff} and by $\bdelta$ the direction of the derivative.  

\begin{definition}[Entropic Derivative]
Let $\bmff$ be a random vector with representation as in \eqref{eq:bmff} and $\bdelta \in \R^{r_1 + r_2}$. Then the entropic derivative of an $\F_T$-measureable random variable $\ell$ in direction $\bdelta$ along the constraint $\bmff$ at time $ \tT$ is
\begin{equation*}
    \EDt[\ell ]
    =
    \lim_{\ep\downarrow0}\frac{1}{\ep}\left(\E^{\Q_\ep}[\ell\;|\;\F_t] -\E[\ell\;|\;\F_t] \right)\,,
\end{equation*}
where for all $\ep$, $\Q^*_\ep$ is the solution to optimisation problem \eqref{opt-epsilon}.
\end{definition}

The entropic derivative is 1-homogeneous, in that $\EDt[ m \ell ] = m \EDt[ \ell ]$ for all $m \in \R$, and additive, i.e. $\EDt[ \ell_1 + \ell_2 ] = \EDt[ \ell_1 ] + \EDt[ \ell_2 ]$, for all $\ell_1, \ell_2$ $\F_T$-measureable random variables. Furthermore, the entropic derivative satisfies $\EDt[ \ell ] = 0$ for all $\tT$, if $\ell$ and $\bmff$ are independent; a result following from the next proposition.

\begin{proposition}
The entropic derivative of $\ell$ in direction of $\bdelta$ along the constraint $\bmff$ at time $\tT$ has representation
\begin{equation*}
    \EDt [\ell]
    =
    \bm C^{-1}\bm\delta \cdot \;\E_t[\left(\bmff-\E[\bmff]\right)\,] \,,
    \quad  t \in [0, T]\,.
\end{equation*}
At time $t=0$, the derivative may be written as
\begin{equation*}
    \EDo [\ell]
    =
    \bm C^{-1}\bm\delta \cdot \E[\left(\bmff-\E[\bmff]\right)\;\left(\ell-\E[\ell]\right)]\,.
\end{equation*}
\end{proposition}

\begin{proof}
Using Theorem \ref{thm:perturbation-mult-const} and in particular the approximation of $\Q^*_\ep$, we have for $\tT$
\begin{align*}
\EDt[\ell] &= \lim_{\ep\downarrow0}\frac{1}{\ep}\left(\E^{\Q^*_\ep}[\ell\;|\;\F_t] -\E_t[\ell] \right)
\\
&=
\lim_{\ep\downarrow0}\frac{1}{\ep}\Big\{\E_t\left[\left(1+\ep\,\bC^{-1} \bdelta\cdot(\bmff-\E[\bmff])\right)\ell\right]-\E_t[\ell]+o(\ep)\Big\}
\\
&= \lim_{\ep\downarrow0}\left\{\bC^{-1} \bdelta\cdot\E_t[(\bmff-\E[\bmff])\;\ell]+O(\ep)\right\}
\\
&= \bm C^{-1}\bdelta \cdot \;\E_t[\left(\bmff-\E[\bmff]\right)\;\ell] \,.
\end{align*}    
The representation for $t = 0$ follows by noting that $\E[(\bmff - \E[\bmff])\E[\ell]\,] = 0$.
\end{proof}
Next, we provide an example of an entropic risk measure and relate it to the sensitivity of the Tail-Value-at-Risk (TVaR) to a sub-portfolio. Recall that for a random variable $Y$ and a probability measures $\Q$, the $\Q$-TVaR at level $\beta \in [0,1)$ is defined as 
\begin{equation*}
    \TVaR^\Q_\beta(Y)  = \frac{1}{1 - \beta}\int_\beta^1 F_Y^{\Q, -1}(u) \diff u\,,
\end{equation*}
where the $\Q$-quantile function of $Y$ is given by $F_Y^{\Q,-1}(u) := \VaR^\Q_u(Y)$. For simplicity of notation, we write $\TVaR_\beta(Y) : =\TVaR^\P_\beta(Y)$

\begin{example}[TVaR Sensitivity]
Suppose we have the process $(X_{1, t},X_{2, t})_{\tT}$ and consider $\mff=\Id_{\{X_{1,T} + X_{2,T}<q\}}$, where $q = \VaR_\alpha(X_{1,T} + X_{2,T}) $, and $\delta \in \R$. For simplicitly assume that the quantile function of $X_{1,T} + X_{2,T}$ is continuous around $\alpha$, then the constraint corresponds to a small perturbation constraint of
\begin{equation*}
    \Q\left(X_{1,T} + X_{2,T} < q\right)
    =
    \P\left(X_{1,T} + X_{2,T} < q\right) + \ep \delta
    = \alpha + \ep \delta\,.
\end{equation*}
Moreover, the entropic derivative of $X_{1,T}$ in direction $\delta$ along $\mff$ at time $\tT$ is
\begin{align*}
     \EDt[X_{1,T} ]
    &= 
    \frac{\delta}{\V\left(\Id_{\{X_{1,T} + X_{2,T} < q\}}\right)} \,\E_t\left[\left(\Id_{\{X_{1,T} + X_{2,T}<  q\}}-\alpha\right) X_{1,T}\right]
    \\
    &= 
    \frac{\delta}{\alpha(1 - \alpha)} \,
    \left((1 -\alpha)\E_t[X_{1,T}] -\E_t[ \Id_{\{X_{1,T} + X_{2,T}\ge  q\}}X_{1,T}] 
    \right)
    \\
    &=
     \frac{\delta}{\alpha} \left(\E_t[X_{1,T}]-
     \frac{1 - \alpha_t}{1 - \alpha}\;
     \E[ X_{1,T}\,|\,X_{1,T}+X_{2,T} \ge q,\, \F_t\,]
     \right)
    \\
    &= 
     \frac{\delta}{\alpha}\, \left(\E_t[X_{1,T}] -
     \frac{1 - \alpha_t}{1 - \alpha}\;
     \frac{\diff}{\diff\ep}\TVaR_\alpha(X_{1,T}(1 + \ep) + X_{2,T}~|~\F_t)\Big|_{\ep = 0}
    \right)\,,
\end{align*}
where $\alpha_t:= \P(X_{1,T} + X_{2,T} < q~|~\F_t)$.
In particular, if $\E[X_{1, T}] = 0$ and for $\delta = \alpha$, we obtain 
\begin{equation*}
    \EDo[X_{1,T} ]
    =
    -\frac{\diff}{\diff\ep}\TVaR_\alpha\left(X_{1,T}(1 + \ep) + X_{2,T}\right)\Big|_{\ep = 0}
    \,.
\end{equation*}
We note that $\frac{\diff}{\diff\ep}\TVaR_\alpha\left(X_{1,T}(1 + \ep) + X_{2,T}\right)\big|_{\ep = 0}$ is the sensitivity of TVaR in direction of sub-portfolio $X_{1, T}$; see e.g. \cite{Hong2009MS}.
\end{example}

Next, we generalise the entropic derivative to the class of distortion risk measures, which subsumes TVaR. First introduced by \cite{Yaari1987Economitrica}, distortion risk measures include a wide range of risk measures used in financial risk management and behavioural economics. 
\begin{definition}[Distortion Risk Measures]
For a function $\gamma \colon [0,1] \to [0, \infty)$ with $\int_0^1 \gamma(u) \, du = 1$, the distortion risk measure of a random variable $Y$ with weight function $\gamma$ under a probability measure $\Q$ is given by
\begin{equation*}
    \rho_\gamma^\Q(Y) 
    : = \int_0^1 F_Y^{\Q, -1}(u)\;\gamma(u) \, du\,.
\end{equation*}
\end{definition}
We set $\rho_\gamma(Y):= \rho_\gamma^\P(Y)$.
A distortion risk measure $\rho_\gamma$ satisfies the properties of coherence if the distortion weight function $\gamma$ is non-decreasing \citep{Kusuoka2001AME}.

\begin{proposition}[Distortion Risk Measures]
Let $\rho_\gamma$ be a distortion risk measure under $\P$ and $\ell$ a absolutely continuous $\F_T$-measurable random variable with support $B$. Then the entropic derivative of the distortion risk measure of $\ell$ in direction $\bm \delta$ along the constraint $\bmff$ at time $t = 0$ is
\begin{align*}
    \EDo[\ell; \gamma] 
    &:=
    \lim_{\ep\downarrow0}\frac{1}{\ep}\left(\rho_\gamma^{\Q_\ep^*}(\ell) -\rho_\gamma(\ell) \right)
    \\
    &=
    - \,\bm C^{-1}\bm\delta \cdot\;\int_B
    \E\left[ (\bmff-\E[\bmff] )\;\Id_{\{\ell \le y\}}\right]
    \; \gamma\left(F_\ell(y)\right)\;\diff y
    \,,
\end{align*}
where $F_\ell(y) := \P(\ell \le y)$.
\end{proposition}

\begin{proof}
Note that the $\Q$-distortion risk measures can be written as an expectation under the reference probability
\begin{equation*}
    \rho_\gamma^\Q(\ell) 
    =
    \E\,\left[F_\ell^{\Q, -1}(U)\;\gamma(U) \right]\,,
\end{equation*}
where $U \stackrel{\P}{\sim} U(0,1)$ is a standard uniform random variable under $\P$. Using this representation, the derivative becomes
\begin{equation*}
    \EDo\;\left[\ell ;\gamma\right]
    =
    \lim_{\ep\downarrow0}\tfrac{1}{\ep}\left(\rho_\gamma^{\Q_\ep^*}(\ell) -\rho_\gamma(\ell) \right)
    =
    \lim_{\ep\downarrow0}\tfrac{1}{\ep}\;\E\left[\left(F_\ell^{\Q^*_\ep, -1}(U) - F_\ell^{ -1}(U)\right)\gamma(U)\right]\,.
\end{equation*}
Next, we calculate the derivative of the $\Q^*_\ep$-quantile function of $\ell$ with respect to $\ep$. For this, note that for all $u \in (0,1)$, differentiating the equation $F_\ell^{\Q^*_\ep}\left(F_\ell^{\Q^*_\ep, -1}(u)\right) = u$ gives
\begin{equation}\label{eq:deriv-F-Q}
    \left.\frac{\diff}{\diff\ep} F_\ell^{\Q^*_\ep, -1}(u)\right|_{\ep = 0}
    =
    - \left.\frac{\frac{\diff}{\diff\ep}F_\ell^{\Q^*_\ep}(x)}{f_\ell(x)}\right|_{\ep = 0, \;x = F_\ell^{-1}(u)}
    \,,
\end{equation}
where $f_\ell(x)$ is the $\P$-density of $\ell$. To simplify \eqref{eq:deriv-F-Q}, we calculate the $\Q^*_\ep$-distribution function of $\ell$ using the approximation of $\frac{\diff \Q^*_\ep}{\diff \P}$ in Theorem \ref{thm:perturbation-mult-const}
\begin{align*}
    F_\ell^{\Q_\ep}(x)
    &=
    \E\left[\Id_{\{\ell \le x\}}\;\left( 1 + \ep\,\bC^{-1} \bdelta \cdot(\bmff-\E[\bmff] )  \right) \right]+o(\ep)
    \\
    &= F_\ell(x) +\ep\,\bC^{-1} \bdelta  \cdot \,\E\left[(\bmff-\E[\bmff] )\;\Id_{\{\ell \le x\}}\right]+o(\ep) \,.
\end{align*}
and thus \eqref{eq:deriv-F-Q} becomes
\begin{equation*}
    \left.\frac{\diff}{\diff\ep} F_\ell^{\Q^*_\ep, -1}(u)\right|_{\ep = 0}
    =
    -\left.\frac{\bC^{-1} \bdelta  \cdot\E\left[(\bmff-\E[\bmff] )\;\Id_{\{\ell \le x\}}\right]}{f_\ell(x)}\right|_{x = F_\ell^{ -1}(u)}\,.
\end{equation*}
Collecting, and using a change of variable $y = F_\ell^{ -1}(u)$ in the second equation, we have
\begin{align*}
    \EDo\;\left[\ell; \gamma \right]
    &=
    - \int_0^1 
    \frac{\bC^{-1} \bdelta  \cdot\E\left[(\bmff-\E[\bmff] )\;\Id_{\{\ell \le F_\ell^{-1}(u)\}}\right]}{f_\ell\left(F_\ell^{ -1}(u)\right)}
    \; \gamma(u)\;\diff u
    \\
    &=
    -\bC^{-1} \bdelta  \cdot  \int_B
    \E\left[(\bmff-\E[\bmff] )\;\Id_{\{\ell \le y\}}\right]
    \; \gamma\left(F_\ell(y)\right)\;\diff y\,.
\end{align*}
\end{proof}

\begin{example}(Sensitivity for Distortion Risk Measures)
Consider the one-dimensional process $X = (X_t)_\tT$ and a constraint $\mff$ satisfying $\mff > 0$ $\P$-a.s. and $\E[\mff] = 1$. Thus, we can write $\tilde{F}_{X_T}(x) := \E[\mff\,\Id_{\{X_T \le x\}}]$ which is a distorted distribution function of $X_T$.

Then the entropic derivative of a distortion risk measure of $X_T$ in direction $\delta$ along the constraint $\mff$ at time $t = 0$ is
\begin{align*}
    \EDo\;\left[X_T; \gamma \right]
    &=
    \frac{- \delta}{\V[\mff]}\int_0^1 
    \frac{\E\left[(\mff-1 )\;\Id_{\{X \le F_{X_T}^{ -1}(u)\}}\right]}{f_{X_T}\left(F_{X_T}^{ -1}(u)\right)}
    \; \gamma(u)\;\diff u
    \\
    &=
    \frac{ \delta}{\V[\mff]}\int_0^1 
    \frac{u - \tilde{F}_{X_T}\left(F_{X_T}^{ -1}(u)\right)}{f_{X_T}\left(F_{X_T}^{ -1}(u)\right)}
    \; \gamma(u)\;\diff u
    \\
    &=
    \frac{ \delta}{\V[\mff]}
    \E\left[
    \frac{F_{X_T}\left(X_T\right) - \tilde{F}_{X_T}\left(X_T\right)}{f_{X_T}\left(X_T\right)}
    \; \gamma(F_{X_T}\left(X_T\right))\;
    \right]
    \\
    &= \frac{ \delta}{\V[\mff]}\,\frac{\diff }{\diff \ep} \rho_\gamma(X_{T,\ep})\Big|_{\ep = 0}\,,
\end{align*}
where the last equality follows from Proposition 4.2 in \cite{Pesenti2021RA} and
where $X_{T,\ep} := F_{\ep}^{-1}(U_{X_T})$ and $U_{X_T}\stackrel{\P}{\sim}U(0,1)$ is a $\P$-uniform comonontonic to $X_T$, i.e. $U_{X_T} := F_{X_T}\left(X_T\right)$, and $F_\ep(x) :=(1-\ep)F_{X_T}(x) + \ep \tilde{F}_{X_T}(x)$, $x \in \R$. We note that $\frac{\diff }{\diff \ep} \rho_\gamma(X_{T,\ep})\Big|_{\ep = 0}$ is the sensitivity of a distortion risk measure to $X_T$ for a perturbation with the mixture distribution $F_\ep$ to $X_T$, see Proposition 4.2 in \cite{Pesenti2021RA}. We refer to \cite{Pesenti2021RA} for a discussion on differential sensitivities to distortion risk measures.
\end{example}

\section{Numerical Example}\label{sec:numerical-ex}
In this section, we illustrate how our methodology may be applied in practice. In particular, for simplicity of exposition, we assume that $X$ is a one-dimensional It\^o process, which more specifically satisfies the SDE under $\P$
\begin{equation}\label{eq:numeric-SDE}
    \diff X_t = \mu(X_t)\,\diff t + \sigma(X_t)\,\diff W_t,
\end{equation}
where $\mu(x)$ and $\sigma(x)$ are parameterised by artificial neural networks. We estimate $\mu$ and $\sigma$ using maximum likelihood estimation (MLE) based on an Euler discretisation of the SDE \eqref{eq:numeric-SDE}. The data is from an automatic marker making (AMM) pool known as sushi-swap for the USDC-WETH cryptocurrency pair, and $X$ represents the price of exchanging one USDC for $X$-WETH. We normalise prices by shifting and scaling them using the mean and standard deviation of the sample path. The data considered is for the entire day of June 29, 2021, and Figure \ref{fig:mu-sigma-crypto} shows the normalised data and the estimated drift and volatility functions. The estimation of $\mu$ and $\sigma$ seen in the figure illustrate that as prices increase, the volatility increases while the drift decreases.
\begin{figure}[h!]
    \centering
    \includegraphics[width=0.6\textwidth]{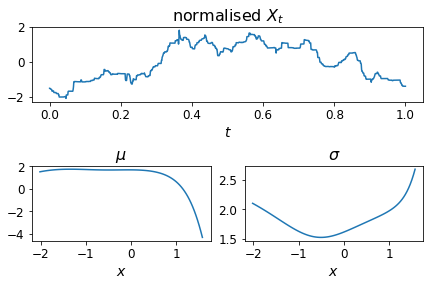}
    \caption{Top panel: Normalised data. Bottom panels: $\mu(x)$ and $\sigma(x)$ estimated from USDC-WETH cryptocurrency AMM pool.}
    \label{fig:mu-sigma-crypto}
\end{figure}

A trader may have a specific view on, e.g., the expected return in the cryptocurrency pool and / or the expected time that prices spend below some level. Using the methodology we developed in this paper, the trader then wishes to update the model estimated on historical data to reflect their beliefs. Hence, with $\mu$ and $\sigma$ estimated (under $\P$), and a specified set of constraints, the trader proceeds to estimate the optimal measure using the algorithmic steps shown in Algorithm \ref{algo:opt-measure}.
\SetKwRepeat{Do}{do}{while}
\begin{algorithm}[htbp]
	\caption{Optimal measure computation.}
	\label{algo:opt-measure}
	\KwIn{drift $\mu$, volatility $\sigma$ functions, constraint functions $\bm f$ and $\bg$ and targets $\bc$ and $\bd$\;}
	
	Initialise $\eone=\bm 0$ and $\etwo=\bm 0$\;
	
	\Do{ $|\bm k(0,x_0) |, |\bm \ell(0,x_0)| > tol$}{

        Solve for $\omega^\dagger(t,x)$ given in \eqref{eq:omega-star} by solving the  PDE
        {
        \[
          \left(\partial_t  + \mu(x)\, \partial_{x} + \tfrac{1}{2} \sigma^2(x)\,\partial_{xx}- \etwo\cdot \bg(x) \right) \omega^\dagger(t,x) = 0,
        \]
        }
        s.t. $\omega^\dagger(T,x)=e^{-\eone\cdot (\bm f(x)-\bc)}$
        using finite-difference (FD) methods\;
        
        Compute (using \eqref{eqn:lambda-h-dagger}) $\lambda^\dagger(t,x) = -\sigma(x)\,\partial_x \log\omega^\dagger(t,x)$ with FD\;
        
        Define $\bm k(t,x):=\E^{\Q^\dagger}[(\bm f(X_T)-\bc)|X_t=x]$. $\bm k(0,x_0)$ gives the terminal constraint errors in \eqref{opt-Z-lambda-h}\;
        
        Solve 
        \[
        \left(\partial_t + (\mu(x)- \sigma(x)\, \lambda^\dagger(t,x))\,\partial_x + \sigma^2(x)\,\partial_{xx}\right) \bm k(t,x) = \bm 0
        \]
        s.t. $\bm k(T,x)=\bm f(x)-\bc$ using FD\;
        
        Define $\bm \ell(t,x):=\E^{\Q^\dagger}[\int_t^T \bg(X_s)\,\diff s|X_t=x]$. $\bm \ell(0,x_0)$ gives the running constraint errors in \eqref{opt-Z-lambda-h}\;
        
        Solve 
        \[
        \left(\partial_t - \sigma(x)\, \lambda^\dagger(t,x)\,\partial_x + \sigma^2(x)\,\partial_{xx}+\bg(x)\right) \bm\ell(t,x) = \bm 0
        \]
        s.t. $\bm \ell(T,x)=\bm 0$ using FD\;
        
        Update $\eone$, $\etwo$ using an optimisation engine
	    }

	\KwOut{$\lambda^\dagger(t,x)$ which is the estimate of $\lambda^*(t,x)$\;}
\end{algorithm}

We consider two numerical examples (i) we increase $\VaR_{0.9}(X_T)$ by 10\% and decrease $\VaR_{0.5}(X_T)$ by 10\%, and (ii)
we increase the $\VaR_{0.9}(X_T)$ by 10\% and reduce the average time spent below the barrier $X_t=-0.1$ by 50\%; all percentages are relative to their values under the reference measure $\P$. The $\VaR_\alpha$ constraints are induced by constraint functions $f(x)=\Id_{\{x<q_\alpha\}}$ with constraint constants $\alpha$, i.e. $\Q(X_T \le q_\alpha) = \alpha$, and where $q_\alpha$ are the $\VaR$ values under $\Q$. The average time time spent below a barrier is achieved by imposing a running cost constraint. To this end, define
\begin{equation*}
\tau = 
\int_0^T g(X_s)\,\diff s\,
\quad \text{with}  \quad g(x) = \Id_{\{x \,\le -0.1\}}.
\end{equation*}
Thus, the constraint 
$$\E^\Q\left[\int_0^T g(X_t)\,\diff t\right] = c$$
corresponds to constraining the average time spent below the barrier $-0.1$ to be equal to $c$. 

We first investigate the case of the two $\VaR$ constraints, example (i). The histogram of $X_T$ under the reference measure $\P$ and the optimal measure $\Q^*$ is show in the left panel of Figure \ref{fig:Two-VaR-stress}. From the left panel, we observe, when comparing the distribution of $X_T$ under $\P$ with $\Q^*$, that probability mass from the centre of the distribution is pushed into the left and right tails to ensure that the median is reduced and the 90\%-quantile is increased. The right panel of Figure \ref{fig:Two-VaR-stress} shows the drift under $\Q^*$. Recall that under $\P$ the drift is a function of the process $X$ only -- see  \eqref{eq:numeric-SDE}; under $\Q^*$, however, the drift depends on the value of the process and on time. From the right panel of Figure \ref{fig:Two-VaR-stress}, we observe that the probability mass transport seen in the histograms of $X_T$ in the left panel, is achieved by having excess positive / negative drift to the right / left of the original median value. Moreover, we see upward / downward spikes at the locations of the new quantiles whose intensity increases as the terminal time approaches.
\begin{figure}[h!]
    \centering
    \begin{minipage}[c]{0.4\textwidth}
    \includegraphics[width=\textwidth]{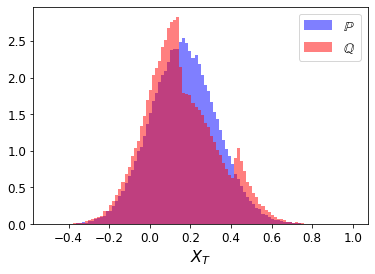}
    \end{minipage}
    \begin{minipage}[c]{0.59\textwidth}
    \includegraphics[width=\textwidth]{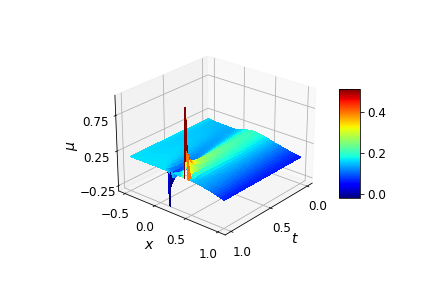}
    \end{minipage}
    \caption{Under $\Q^*$, we impose that VaR at levels $\alpha=0.5$ and $\alpha = 0.9$ are decreased and increased by 10\%, respectively. Percentage changes are relative to the corresponding values under $\P$. Left panel: histogram of $X_T$ under $\P$ (blue) and under $\Q^*$ (red). Right panel: drift of $X$ under $\Q^*$.}
    \label{fig:Two-VaR-stress}
\end{figure}

Next, we next investigate the case of a 10\% increasing in the 90\%-quantile and a 50\% reduction the average time spent below the barrier, i.e., example (ii). The right top panel of Figure \ref{fig:VaR-Barrier-stress} displays the  histogram of $\tau$ under the reference $\P$ and optimal measure $\Q^*$. The figure shows that under $\Q^*$, the amount of time spent below the barrier is more concentrated towards zero than it is under the reference measure. The top left panel shows the histogram of $X_T$, and while under $\Q^*$ the 90\%-quantile is increased, which is seen by the additional mass in the right tail, the running cost constraint moves mass away from the left tail. The bottom panel of the figure shows the drift under $\Q^*$. We observe that as the process crosses to negative values, it receives a positive drift which prevents the process from spending additional time below the barrier. The process also receives a drift if it approaches the target quantile whose intensity increases as the terminal time approaches.
\begin{figure}[h!]
    \centering
    \includegraphics[width=0.4\textwidth]{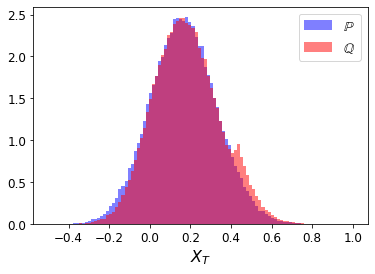}
    \qquad
    \includegraphics[width=0.4\textwidth]{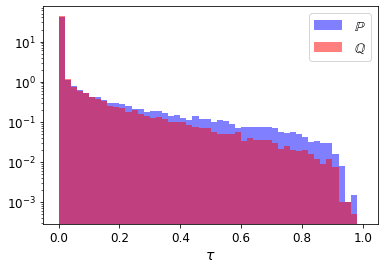}    
    \\
    \includegraphics[width=0.6\textwidth]{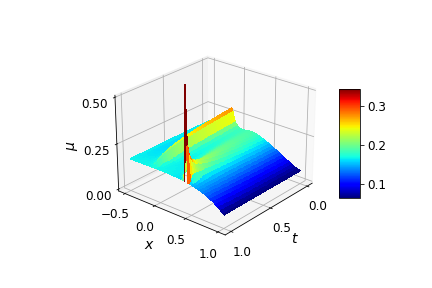}
    \caption{Under $\Q^*$, we impose that VaR at level $\alpha=0.9$ is increased by 10\% and that the average time spent below the barrier level $-0.1$ is decreased by 50\%. Percentage changes are relative to the corresponding values under $\P$. Under the $\P$-measure, the process spends about 4\% of the time below the level $-0.1$. Top left panel: histogram of $X_T$ under $\P$ (blue) and under $\Q^*$ (red). Top right: histogram of the time spent below the barrier $\tau$ under $\P$ (blue) and under $\Q^*$ (red). Bottom panel: drift of $X$ under $\Q^*$.}
    \label{fig:VaR-Barrier-stress}
\end{figure}




\begin{acks}[Acknowledgments]
\end{acks}

\begin{funding}
SJ and SP acknowledge support from the Natural Sciences and Engineering Research Council of Canada (grants RGPIN-2018-05705, RGPAS-2018-522715, and DGECR-2020-00333, RGPIN-2020-04289). We are grateful to L.~P.~Hughston for helpful comments.
\end{funding}

\bibliographystyle{imsart-number} 
\bibliography{references.bib}

\end{document}